\documentclass{amsart}
\usepackage[margin=3cm]{geometry}
\usepackage{mathtools}
\usepackage{amsfonts}
\usepackage{amsthm}
\usepackage{amssymb}
\usepackage{enumitem}
\usepackage{amsthm, amsmath, hyperref, dsfont, enumitem, tcolorbox, esint, mathtools,contour}
\usepackage{tikz}
\usepackage{tikz-cd}
\usepackage{enumitem}
\mathtoolsset{showonlyrefs=true}

\newcommand{\Pt}{P^{\ell}_{\tau}}

\theoremstyle{plain}
\newtheorem{theorem}{Theorem}[section]
\newtheorem{proposition}[theorem]{Proposition}
\newtheorem{lemma}[theorem]{Lemma}
\newtheorem{corollary}[theorem]{Corollary}

\theoremstyle{definition}
\newtheorem{definition}[theorem]{Definition}

\theoremstyle{remark}
\newtheorem{remark}[theorem]{Remark}

\numberwithin{equation}{section}

\DeclareMathOperator{\Hol}{Hol}
\DeclareMathOperator{\Graph}{Graph}
\DeclareMathOperator{\Hom}{Hom}
\DeclareMathOperator{\Sec}{Sec}

\tikzset{
    rp1_circle/.style={thick, black},
    root_f/.style={circle, fill=blue, inner sep=1.8pt, draw=blue!50!black},
    root_g/.style={rectangle, fill=green!70!black, inner sep=2.2pt, draw=green!50!black},
    move_arrow/.style={->, thick, dashed, blue!70!black}
}

\newcommand{\tb}[1]{#1}

\usepackage{todonotes}

\newcommand{\Gr}{\mathrm{Gr}}

\newcommand{\bcol}{\begin{pmatrix}}
\newcommand{\ecol}{\end{pmatrix}}

\title[h-principle for distributions of corank two of odd rank]{h-principle for corank two distributions of odd rank and maximal first Kronecker index}
\thanks{ I.\ Zelenko was partly supported by NSF grant DMS 2105528 and Simons Foundation Collaboration Grant for Mathematicians 524213.}
\author{Milan Jovanovic}
\address{Milan Jovanovic\\ Department of Mathematics\\
         Texas A\&M University\\
         College Station\\
         Texas \ 77843\\
}
\email{milankj@tamu.edu}

\author{Igor Zelenko} 
\address{Igor Zelenko\\
         Department of Mathematics\\
         Texas A\&M University\\
         College Station\\
         Texas \ 77843\\
         USA}
\email{zelenkotamu@tamu.edu}
\urladdr{\url{https://people.tamu.edu/~zelenkotamu/}}

\begin{document}

\begin{abstract}
We study corank-two distributions $D$ of odd rank $2N+1$ whose associated pencil of
skew-symmetric forms $\{d\alpha|_{D}(q)\mid \alpha\in\Omega^1(M),\ \alpha|_{D}=0\}$
lies, at every point, in the generic orbit of the natural
$\mathrm{GL}\bigl(D(q)\bigr)$-action; equivalently, $D$ has maximal first Kronecker
index. In corank two this class is the analogue of contact and even-contact
distributions. We prove that the corresponding differential relation is ample and
hence satisfies the multi-parametric $C^0$-close $h$-principle, and we deduce a
complete topological characterization: for $N\ge1$, a $(2N+3)$-dimensional manifold
carries such a distribution if and only if
\[
TM\cong\bigl(A\otimes\mathrm{Sym}^{N}(Q^{*})\bigr)\oplus\bigl(A^{*}\otimes\mathrm{Sym}^{N-1}(Q)\otimes\Lambda^{2}Q\bigr)\oplus Q
\]
for some rank two bundle $Q$ and some real line bundle $A$; for a given distribution one can take $Q=TM/D$ and $A=\det D$. The criterion is satisfied by non-parallelizable manifolds as
well. To the best of our knowledge, ampleness has been verified so far for two classes of distributions defined by an
open $\mathrm{Diff}$-invariant relation: even-contact distributions (McDuff, 1987) and distributions of rank greater than
two with maximal small growth vector, in arbitrary ambient dimension
(Mart\'inez-Aguinaga, 2026); for $N\ge2$ our class is a proper open subclass of the latter, singled out by a further condition on the pencil.
 The verification of ampleness reduces, via Kronecker--Weierstrass normal forms
of the restrictions of the pencil to hyperplanes, to computing the convex hulls of the
connected components of the complement of the resultant hypersurface in the space of
pairs of real binary forms of fixed degrees, components separated by a winding number,
i.e. by the classical Cauchy index. This is in contrast with the even-contact case of
McDuff, where the set to be removed is thin, i.e. of codimension at least two, so that
its complement is automatically connected and ampleness is immediate. For completeness, we also show that for corank two distributions on even-dimensional manifolds the regularity of the pencil defines an open ample relation, while the finer open relations prescribing the number of real roots of the Pfaffian of the pencil (for $(4,6)$-distributions, the hyperbolic and elliptic ones) are never ample.

	\end{abstract}
\maketitle
\section{Introduction}

A rank $\ell$ distribution $D$ on a manifold $M$ is a rank $\ell$ subbundle of its tangent bundle.
The corank of such a distribution is by definition $\dim M-\ell$.
A general question of differential topology is whether a distribution satisfying a Diff-invariant
differential relation $\mathcal{R}$ exists on a given manifold.
Examples extensively treated in the literature include contact distributions
\cite{bennequin1983, eliashberg1989, bem2015}, even-contact distributions
\cite{mcduff1987}, and Engel distributions \cite{casals2017, casals2020,
	delpinovogel2020}.
Gromov \cite{Gromov86} proposed to study this kind of question by
trying to prove or disprove properties that are now called the \emph{Gromov
	homotopy principle} or, shortly, the \emph{$h$-principle}.
He developed a variety of methods for proving this principle, such as holonomic
approximation and convex integration (see also \cite{EliashbergMishachev2002}).

Let us briefly recall the $h$-principle.
Consider a fiber bundle $\pi \colon X \to M$ over $M$.
The $r$-\emph{jet bundle} $X^{(r)}$ is the bundle over $M$ with the
fiber $X^{(r)}_p$ over a point $p\in M$ consisting of all $r$-jets of local
sections of the bundle $X$ at $p$.
For each integer $r\geq 0$, there is a natural injective map $J^r$, called the
$r$-\emph{jet extension}, mapping a section $f$ of $X$ to a section $J^rf$
of $X^{(r)}$ so that $J^rf(p)$ is the $r$-jet of $f$ at $p$.
A section $M \to X^{(r)}$ is said to be \textit{holonomic} if it is in the image of $J^r$.
The set of all holonomic sections will be denoted $\Hol(X^{(r)})$.

A \emph{differential relation $\mathcal{R}$ of order r} is a subset of $X^{(r)}$.
A \textit{formal solution} of $\mathcal{R}$ is a section $F \colon M \to X^{(r)}$
for which $\Graph F \subseteq \mathcal{R}$.
The set of formal solutions is denoted $\Sec \mathcal{R}$.
A formal solution $F$ is said to be a \textit{genuine solution} if it is holonomic,
i.e. $F = J^r(f)$ for some $f \colon M \to X$.
The set of genuine solutions is denoted $\Hol \mathcal{R}$.


A differential relation $\mathcal{R}$ is said to satisfy a \emph{$0$-parametric
	$h$-principle} if any formal solution can be continuously deformed (homotoped)
in $\mathrm{Sec}(\mathcal{R})$ to a genuine solution.
This is equivalent to the surjectivity of the map
\begin{equation}
	\label{induced}
	i_*\colon \pi_0(\mathrm{Hol}(\mathcal{R})) \rightarrow \pi_0(\mathrm{Sec}(\mathcal{R}))
\end{equation}
where $i\colon \mathrm{Hol}(\mathcal{R}) \hookrightarrow
	\mathrm{Sec}(\mathcal{R})$ represents inclusion.
A differential relation $\mathcal{R}$ is said to satisfy a \emph{$1$-parametric
	$h$-principle} if any $1$-parametric family of formal solutions $\{f_t\}_{t
	\in [0,1]}$ joining two genuine solutions can be deformed inside
$\mathrm{Sec}(\mathcal{R})$, keeping $f_0, f_1$ fixed, into a family
$\{f_t'\}_{t \in [0,1]}$ of genuine solutions of $\mathcal{R}$.
In particular, it implies that the map  $i_*\colon
	\pi_0(\mathrm{Hol}(\mathcal{R})) \rightarrow \pi_0(\mathrm{Sec}(\mathcal{R}))$
is one-to-one, so, together with the $0$-parametric h-principle, this map is
bijective (the $1$-parametric $h$-principle is in fact stronger than injectivity of $i_*$, since the deformation is required to be relative to the endpoints).
More generally, a differential relation $\mathcal{R}$ is said to satisfy the \emph{multi-parametric $h$-principle}
if every family of formal solutions parametrized by a disk $\mathbb D^k$, $k\ge0$, which consists of genuine solutions over $\partial\mathbb D^k$, can be deformed, relative to $\partial\mathbb D^k$, inside $\mathrm{Sec}(\mathcal R)$ to a family of genuine solutions. This implies that the inclusion $i$ induces isomorphisms of all homotopy groups, i.e. $\mathrm{Hol}(\mathcal{R})$
and $\mathrm{Sec}(\mathcal{R})$ are weakly homotopy equivalent (see \cite[Section 6.2]{EliashbergMishachev2002}).

All versions of the $h$-principle hold for open diffeomorphism-invariant
differential relations on open manifolds, thanks to the Holonomic Approximation
Theorem \cite{Gromov86, EliashbergMishachev2002}.
For closed manifolds,
establishing even the $0$-parametric $h$-principle is usually a very
sophisticated task requiring a case-by-case treatment for ``filling balls''.
Furthermore, even if the $0$-parametric $h$-principle holds, the induced map
$i_*$ is typically surjective but not injective.
Generally, one must restrict to a special subclass of genuine solutions, such
as overtwisted contact structures \cite{eliashberg1989, bem2015}, or loose
\cite{casals2020} and overtwisted \cite{delpinovogel2020} Engel structures, so
that the restriction of $i_*$ to this subclass becomes injective.
In existing examples, all
versions of the $h$-principle hold when restricting to these flexible
subclasses.

The applicability of the Gromov convex integration method \cite{Gromov1973} is
more restrictive than that of the Holonomic Approximation Theorem.
However, if its conditions hold, then convex integration automatically yields
the multi-parametric
$h$-principle, even for closed manifolds.
Moreover, in contrast to the Holonomic Approximation Theorem (which lacks
$C^0$-proximity), convex integration establishes the $C^0$-close
multi-parametric $h$-principle.
Here ``$C^0$-close'' refers to the underlying sections of the bundle $X$ (in our case, to the distributions themselves, and not to their $1$-jets): the homotopy from a formal solution $F$ to a genuine solution $J^r f$ can be chosen so that the underlying sections of $X$ of all the formal solutions in the homotopy, and in particular $f$ itself, lie in an arbitrarily small $C^0$-neighborhood of the underlying section of $F$. In particular, the underlying sections of genuine solutions are $C^0$-dense among the underlying sections of formal solutions. (No such statement holds for the derivatives: the derivatives of the genuine solutions produced by convex integration oscillate rapidly and are close to those of the formal solution only in an averaged sense.)

Returning to distributions, given a vector bundle $V$ over $M$
with a fiber $V_q$ over $q \in M$, denote by $\mathrm{ Gr}_k(V)$ the Grassmannian
bundle of $V$.
With this notation, distributions are smooth sections of the bundle
\begin{equation}
	\label{X_dist}
	X=  \mathrm{Gr}_\ell(TM)\cong \mathrm{Gr}_{\dim M-\ell} (T^*M).
\end{equation}

To the best of our knowledge, in the case of distributions satisfying certain
open differential relations, ampleness (and hence the applicability of convex integration in its classical form) has so far been verified in two cases:

\begin{enumerate}
	\item Multi-parametric $C^0$-close $h$-principle for even-contact
	      distributions was proved by D. McDuff \cite{mcduff1987}.
	      Recall that a corank one distribution $D$ of odd rank is called
	      \textit{even-contact} if for any defining $1$-form $\alpha$ of $D$ (i.e., such
	      that $D=\mathrm{ker}\, \alpha$), the  $2$-form $d\alpha|_D$ has a
	      $1$-dimensional kernel at every point of $M$, or, equivalently, if
	      $\mathrm{rank}\, D=2N+1$, then $\alpha \wedge d\alpha ^N\neq 0$.
	      Clearly, this is an open diffeomorphism-invariant differential
	      relation on the first jet bundle of the Grassman bundle $\mathrm
		      {Gr}_\ell(TM)$, where $\ell=\dim M-1$;

	\item Multi-parametric $C^0$-close  $h$-principle  was proved for all
	      distributions of rank higher than two with the maximal possible  small
	      growth vector at every point in
	      \cite{martinez2026}.
	      Recall that given a distribution $D$, the
	      \emph{weak $j$th power $D^j$} of the distribution $D$  is recursively
	      defined by $$D^j:=D^{j-1}+[D, D^{j-1}], \quad D^1:=D$$ and the
	      \emph{small growth vector} at a point $q\in M$  is the tuple $(\dim\,
		      D(q), \dim~D^2(q), \dim D^3(q), \cdots)$.
\end{enumerate}
Note that in the case of corank one distributions of odd rank (resp., even rank), the condition of  being even-contact (resp., contact) is much more subtle than the condition of being just of maximal small growth vector.
We also mention the recent work \cite{MAdP2021}, where a refinement of convex integration (``convex integration with avoidance'') is developed and applied to hyperbolic $(4,6)$-distributions, i.e. to a differential relation which is \emph{not} ample (see also Subsection \ref{even_dim_sec} below); the results of the present paper do not use this refinement.

In the present paper, we study $h$-principle for a class of distributions of
corank two and odd rank, which is a natural analog of contact and even-contact
distributions in this rank and corank.
To describe this class and justify this analogy, let us give a more general
construction of a special
subspace of skew
symmetric forms associated at a given point to a distribution $D$.
For any non-trivial defining $1$-form $\alpha$ of $D$, i.e. such that
$\alpha|_D=0$, consider the $2$-form $d\alpha|_D$ on $D$.
Observe that for a given $q\in M$, the form $d\alpha|_D(q)$
depends only on the values of the $1$-form  $\alpha$ at $q$, since
\begin{equation}
	\label{Cartan}
	d\alpha|_D(q)(X, Y)=-\alpha(q)([\underline{X}, \underline {Y}](q)), \quad X,Y\in D(q),
\end{equation}
where $\underline {X}, \underline {Y}$ are local sections of $D$ and equal $X$ and $Y$ at $q$.
Let $\mathcal{D}^\perp \subset T^*M$ be the  \emph{annihilator bundle of $D$},
$$D^\perp(q)=\{\alpha \in T^*M: \alpha|_D(q)=0\}.$$
By \eqref{Cartan}, the map
\begin{equation}
	\label{Pi_2}
	L_q:D^\perp(q)\to \wedge^2 D(q)^*,\quad
	\alpha\mapsto d\alpha|_D(q)
\end{equation}
is  a well-defined linear map. The map $L_q$ is dual to the Lie
bracket map (or \emph{Levi map}) \( [\,\cdot\,, \,\cdot\,]_q \colon D(q) \wedge
D(q) \longrightarrow T_qM / D(q) \) (see \cite {KrynskiZelenko2011}). This motivates the following

\begin{definition}
	\label{pencil_def}
	The map $L_q$, as in  \eqref{Pi_2}, is called the \emph{dual Levi map} associated with a distribution $D$ at a point $q$.
	The subspace $\mathrm {Im}\, L_q$ in the space $\wedge^2 D(q)^*$ of skew-symmetric forms on $D(q)$
	is called the \emph{pencil of skew-symmetric forms} associated with the 
	distribution $D$ at a point $q\in M$ \footnote{In linear algebra, by
		pencils one usually means planes of objects (matrices, forms, linear
		operators, etc.) Here we use  the word ``pencil" in a more general
		sense as $\dim \mathrm{Im} L_q$ is in general not equal to two.}
\end{definition}
\begin{remark}
	\label{rem:dual-levi-well-defined}
	Suppose $D$ and $\widehat D$ are local distributions defined in a neighborhood of $q$,
	and $D(q) = \widehat D(q)$.
	One can check that if the 1-jets $J^1(D)(q)$ and $J^1(\widehat D)(q)$ coincide at $q$,
	then the dual Levi maps $L_D$ and $L_{\widehat D}$ coincide at $q$.
\end{remark}
The map $L_q$ is injective if and only if
\begin{equation}
	\label{step2}
	D^2(q)=T_qM,
\end{equation}
i.e. the distribution $D$ is step $2$ at $q$. In this case $\dim \mathrm{Im} L_q=\mathrm{corank} D$.

Given a point $q \in M$, the standard action of the general linear group
$\mathrm{GL}(D(q))$ on $D(q)$ induces a natural action on the space  $\wedge^2 D(q)^*$
of skew-symmetric forms on $D(q)$, and therefore also on $\Gr_2(\wedge^2 D(q)^*)$.
In the case of distributions of corank one and odd rank (resp., of even rank), the condition of being even-contact (resp., contact) at $q$ is equivalent to the fact that $\mathrm{Im}\,L_q$ is a line and this line 
belongs to the generic (non-empty Zariski open) 
orbit in the projective space of $\wedge^2 D(q)^*$ under this natural action of $\mathrm{GL}(D(q))$.

In the case of corank two distributions of step $2$, i.e. when \eqref{step2}
holds, $\mathrm{Im}\,L_q$ is a plane (or a pencil) in $\wedge^2 D(q)^*$. The
classification of orbits of planes in the space $\wedge^2 V^*$ of
skew-symmetric forms of a vector space $V$ under the natural
$\mathrm{GL}(V)$--action was given by M. Gauger \cite{gauger1973}, based on the
classical Kronecker--Weierstrass theory of pencils of matrices \cite[Chapter
	XII]{Gantmacher}, by prescribing a canonical representative (the canonical
normal form) for each orbit\footnote{To be precise, the classical settings
	of \cite{Gantmacher} and \cite{gauger1973} give a canonical decomposition for a fixed choice
	of a pair of forms, whereas we are interested in their linear span; i.e., a change of
	basis within the plane of chosen forms is allowed. The classification with
	respect to this slightly larger group of transformations can be easily obtained
	from the classical one.}.
When the dimension of $V$ is odd, there exists a generic (non-empty Zariski open)
orbit in the Grassmannian $\mathrm{Gr}_2(\wedge^2 V^*)$ of planes in $\wedge^2 V^*$ under the natural $\mathrm{GL}(V)$--action \footnote{In contrast, if the dimension of $V$ is even and at least $8$, there are no open orbits in $\mathrm{Gr}_2(\wedge^2 V^*)$ due to the existence of continuous invariants.}.

This generic orbit can be described in several equivalent ways.
First, since $\dim V$ is odd,
every skew-symmetric form on $V$ in this pencil has a nontrivial kernel. Furthermore, for any plane $\mathcal P$ in $\wedge^2 V^*$
there exists a homogeneous (vector-valued) polynomial map $B \colon \mathcal P \to V$ such that for any nonzero  $\omega$ in $\mathcal P$, $B(\omega) \in \ker \omega$  and $B(\omega) \neq 0$. The \emph{first minimal index} (or \emph{first Kronecker index}) of the pencil $\mathcal P$
is by definition the minimal possible degree of such a polynomial map. The following proposition is a direct consequence of the classification of pencils of skew symmetric forms, given in \cite[Theorem 6.8, p. 309]{gauger1973}. For some items we need to consider a complexification $\mathcal P^\mathbb C$ of the plane $P$. While $P^\mathbb C$ is by definition a span over $\mathbb C$ of $\mathcal P$, each element of it can be seen as an element of $\wedge^2V^\mathbb C$, where $V^\mathbb C$ is the complexification of $V$.

\begin{proposition}
	\label{open_orbit_char_prop}
	Let $V$ be an odd-dimensional vector space, $\dim V=2N+1$. A plane $\mathcal P$ in $\wedge^2 V^*$ belongs to the generic orbit in the Grassmannian  $\mathrm{Gr}_2(\wedge^2 V^*)$  under the natural $\mathrm{GL}(V)$--action if one of the following three equivalent conditions holds:

	\begin{enumerate}
		\item There exists a basis $(\varepsilon_1, \cdots,\varepsilon_{2N+1})$ of the dual space $V^*$ such that
		      \begin{equation}
			      \label{singular block}
			      \begin{split}
				      \mathcal P=\mathrm{span}\left\{\varepsilon_{1}\wedge \varepsilon_{2N+1} + \dots + \varepsilon_{N} \wedge \varepsilon_{N+2}, \,   \varepsilon_{2} \wedge \varepsilon_{2N+1} + \dots + \varepsilon_{N+1} \wedge \varepsilon_{N+2}\right\}= \\
				      \{(\mu \varepsilon_1+\lambda \varepsilon_2)\wedge \varepsilon_{2N+1}+(\mu \varepsilon_2+\lambda \varepsilon_3)\wedge \varepsilon_{2N}+\cdots+
				      (\mu \varepsilon_N+\lambda \varepsilon_{N+1})\wedge \varepsilon_{N+2}\mid \lambda, \mu\in \mathbb R\}.
			      \end{split}
		      \end{equation}
		      In other words, the Kronecker--Weierstrass canonical form for the pencil $\mathcal P$ consists of exactly one singular block.
		\item
		      Any nonzero form  $\omega  \in \mathcal P^\mathbb C$ has  $1$-dimensional kernel,
		      \begin{equation}
			      \label{1dim_ker}
			      \dim\, \mathrm {ker}\,\omega=1, \quad \forall \omega\in \mathcal P^\mathbb C\setminus \{0\}.
		      \end{equation}
		      or, equivalently,
		      \begin{equation}
			      \label{wedge_power}
			      \omega^N\neq 0\quad \forall \,\omega\in \mathcal P^\mathbb C \setminus \{0\}.
		      \end{equation}
		\item The first minimal index of the pencil $\mathcal P$ is equal to $N$ (which is the maximal possible value for the first minimal index in this dimension, so the first minimal index is also the only minimal index).
		\item For some (and therefore any) basis $\omega_1$ and $\omega_2$ of $\mathcal P^\mathbb C$ the forms $\{\omega_1^i\wedge
			      \omega_2^{N-i}\}_{i=0}^N$ are linearly independent.


	\end{enumerate}
\end{proposition}

\medskip




For example, to understand why item  (1)  of Proposition \ref{open_orbit_char_prop} implies  items  (2), (3), and (4), note that if $\omega$ is a skew-symmetric $2$-form on $V$ such that $\omega^N\neq 0$, then $X$ generates the kernel of $\omega$ if and only if
\begin{equation}
	\label{volume_ker}
	\omega ^N=\iota_X\Omega
\end{equation}
for a volume form $\Omega $ on $V$. Consequently, if $(e_1, \cdots, e_{2N+1})$ is a basis of $V$ dual to $(\varepsilon_1, \cdots,\varepsilon_{2N+1})$, then from \eqref{volume_ker} it is easy to show that
\begin{equation}
	\label{kernel_coordinate}
	\begin{split}
		~ &
		\mathrm{ker}\,(\{(\mu \varepsilon_1+\lambda \varepsilon_2)\wedge \varepsilon_{2N+1}+(\mu \varepsilon_2+\lambda \varepsilon_3)\wedge \varepsilon_{2N}+\cdots+
		(\mu \varepsilon_N+\lambda \varepsilon_{N+1})\wedge \varepsilon_{N+2})=
		\\~&
		\mathrm{span}\{\sum_{i=1}^{N+1}(-1)^i\lambda^{N+1-i} \mu^{i-1}e_i\},
	\end{split}
\end{equation}
which implies items (2), (3), and (4).
\medskip

Now we are ready to formulate our differential relation on the first jet bundle $X^{(1)}$ of the bundle $X=\mathrm {Gr}_{2N+1}(TM)\cong \mathrm{Gr}_2(T^*M)$:
\medskip

\textbf{Relation $\mathfrak R$:} \emph{A 1-jet $[D]_q \in X_q^{(1)}$ at a point $q$ is said
to satisfy the differential relation $\mathfrak R_q$ if the image of the dual Levi map $\mathrm{Im}\, L_q$
is two-dimensional and belongs to the generic (non-empty  Zariski open)
orbit in the Grassmannian $\mathrm{Gr}_2(\wedge^2 D(q)^*)$ of planes in
$\wedge^2 D(q)^*$ under the natural $\mathrm{GL}\bigl(D(q)\bigr)$--action.
Equivalently, if any one of the four conditions from Proposition
\ref{open_orbit_char_prop} holds (with $V$ replaced by $D(q)$).}
\medskip

The genericity of the orbit in relation $\mathfrak R$ justifies the statement that the class of corank $2$ distributions of rank $2N+1$ satisfying condition $\mathfrak R$ is a natural analogue of even-contact and contact distributions within the context of the considered rank and corank.

Our main result is the following:

\begin{theorem}
	\label{main_theorem}
	The differential relation $\mathfrak R$ satisfies the multi-parametric $C^0$-close $h$-principle.
\end{theorem}

Note that for $N\ge2$ the condition we impose in relation $\mathfrak R$ is much stronger than just the maximality of the small growth vector; therefore, in this range, Theorem \ref{main_theorem} does not follow from the results of \cite{martinez2026}. Indeed, in our setting the maximality of the small growth vector simply means that the distribution is of step $2$, i.e. that for every $q$ the space $\mathrm{Im}\,L_q$ is exactly $2$-dimensional, but this does not force it to belong to the generic orbit in $\mathrm{Gr}_2(\wedge^2 D(q)^*)$ (for instance, a pencil spanned by two decomposable forms with a common factor is not in the generic orbit for $N\ge2$). Thus $\mathfrak R$ is a proper open subrelation of the relation considered in \cite{martinez2026}: every formal solution of $\mathfrak R$ is in particular a formal solution of the latter, whereas a genuine solution of the latter need not satisfy our condition at any point. For $N=1$ the two relations coincide: every nonzero skew-symmetric form on a $3$-dimensional space has a $1$-dimensional kernel, and, since the map $\omega\mapsto\ker\omega$ identifies $\wedge^2 D(q)^*$ with $D(q)\otimes\Lambda^3D(q)^*$, the kernels of the forms of any $2$-plane in $\wedge^2D(q)^*$ span a $2$-dimensional subspace, so that every $2$-plane belongs to the generic orbit; in this case Theorem \ref{main_theorem} recovers the case of rank $3$ distributions of maximal small growth vector on $5$-manifolds treated in \cite{martinez2026}.

The proof of Theorem \ref{main_theorem} consists of checking that the
differential relation $\mathfrak R$ is ample (see
Section \ref{ample_sec} for the definition of ampleness). Then the statement
follows from the general theorem of Gromov \cite{Gromov1973} that  any ample open
differential relation satisfies the multi-parametric $C^0$-close $h$-principle (see
Theorem \ref{Gromov_ample_thm} below).

The proof of ampleness for even-contact distributions is based on the
elementary linear-algebraic version of the Darboux theorem for skew-symmetric
forms. It leads to the fact that the intersection of the defining differential relation
with any principal direction is the complement of a
thin (i.e., codimension 2) set.
In our case, the ampleness check is based on
the Kronecker-Weierstrass normal forms of the restrictions of the pencil (called \emph{reduced pencils} in Subsection \ref{preliminaries}) to all
possible hyperplanes. This is significantly more involved both because there
are more possibilities to consider and because the intersection of the
relation with certain principal subspaces is the complement of a codimension 1
subset.
This involves classical algebraic geometry, as the intersection with certain principal relations is, in essence, the complement of the zero locus of the resultant of two binary forms of given degrees with real coefficients (see Corollary~\ref{cor:reduction-to-resultant} below).


We complete the Introduction by formulating  the topological criterion for a manifold to admit a formal solution
of the differential relation $\mathfrak R$. We start by examining in greater detail what it means for a manifold to admit such a solution. First, note that the relation $\mathfrak{R}$ is the preimage of a relation on a natural ``intermediate'' bundle $\widetilde{X}$,
\begin{equation}
	\label{intermediate_bundle}
	X^{(1)} \xrightarrow{\Pi} \widetilde{X} \to X\to M,
\end{equation}
described as follows:
Given $D \in X_q = \mathrm{Gr}_{2N+1}(T_q M)$, the fiber $\widetilde{X}_{D}$ of $\widetilde{X}$ over $D$ is given by $\mathrm{Hom}\bigl(D^\perp, \bigwedge^2 D^*\bigr)$,
\begin{equation}
	\label{tildefiber}
	\widetilde{X}_{D}\cong \mathrm{Hom}\bigl(D^\perp, \textstyle{\bigwedge^2} D^*\bigr).
\end{equation}
To define the canonical map $\Pi: X^{(1)} \to \widetilde{X}$, given $y \in
	X^{(1)}_{D}$, choose a local distribution $\widehat{D}$ (i.e., a section of
$X$) whose first jet represents $y$. By Remark
\ref{rem:dual-levi-well-defined}, the dual Levi map associated with
$\widehat{D}$ at $q$ (see Definition~\ref{pencil_def}) depends only on $y$ and
not on the choice of the local distribution $\widehat{D}$ representing $y$;
thus, we get a well-defined map $\Pi: X^{(1)} \to \widetilde{X}$. Furthermore, it is
not hard to show (see Proposition  \ref{Pi_surjective}) that the map $\Pi$ is surjective, i.e. any element of
$\mathrm{Hom}\bigl(D^\perp, \bigwedge^2 D^*\bigr)$ can be realized as the dual
Levi map associated with a local distribution $\widehat{D}$ satisfying
$\widehat{D}(q) = D$.

A formal solution exists if and only if there is a section of an intermediate bundle $\widetilde{X} \to M$
with a graph belonging to $\Pi(\mathfrak{R})$, where $\Pi \colon X^{(1)} \to \widetilde{X}$ is the canonical projection. The ``only if'' part is obvious, as such a section is a composition of $\Pi$ with a formal solution of $\mathfrak{R}$. For the ``if'' part note that, by the definition of the relation $\mathfrak{R}$, the whole fiber of the bundle $\Pi \colon X^{(1)} \to \widetilde{X}$ over a point in $\Pi(\mathfrak{R})$ is contained in $\mathfrak{R}$; moreover, by Proposition \ref{Pi_surjective} below, $\Pi$ is fiberwise affine and surjective, so that the fibers of $\Pi$ over the graph of a given section of $\widetilde X$ are nonempty affine spaces, and a global lift of the section to a section of $X^{(1)}$ is obtained by patching local lifts with a partition of unity.

Although it is clear from the definition what it means for a section of $\widetilde{X} \to M$ to have a graph in $\Pi(\mathfrak{R})$, we still would like to write it down explicitly, as it is an analogue of an almost contact and almost even-contact structure in this rank and corank setting:

\begin{definition}
	\label{almost_Kronecker_def}
	Suppose $\dim M=2N+3$.
	A section of the bundle  $\widetilde X\to M$ whose graph is in  $\Pi (\mathfrak R)$  is called an \emph{almost Kronecker structure} on $M$.
	It  consists of the following data:
	\begin{enumerate}
		\item A smooth distribution  $D \subset TM$ of odd rank $2N+1$ and corank 2.
		\item A smooth vector bundle homomorphism
		      \[
			      L \colon {D}^\perp \longrightarrow \wedge^2 {D}^*
		      \]
		      which is fiberwise injective.
		\item The pair $({D}, L)$ satisfies the following algebraic fiber constraint at every point: for every $q \in M$ the pencil $L\bigl(D^\perp(q)\bigr)$ belongs to the generic  orbit in the Grassmannian $\mathrm{Gr}_2(\wedge^2 D(q)^*)$ of planes in $\wedge^2 D(q)^*$ under the natural $\mathrm{GL}\bigl(D(q)\bigr)$--action, characterized by one of the four equivalent conditions from Proposition \ref{open_orbit_char_prop} (with $V$ replaced by $D(q)$).
	\end{enumerate}
\end{definition}

Clearly, \emph{every parallelizable odd-dimensional manifold of dimension at
	least $5$ admits an almost Kronecker structure}.
Indeed, start by taking a global frame
$(e_1, \ldots, e_{2N+3})$ on $M$.
Form a distribution $D$ spanned by the
first $2N+1$ vector fields of the frame.
Take $L$ to be any linear map on $D^\perp$ whose image is
given by \eqref{singular block},
where $\varepsilon_1, \ldots, \varepsilon_{2N+3}$ is the dual coframe and
$\lambda$ and $\mu$ are coordinates on $D^\perp$ with respect to the basis
$(\varepsilon_{2N+2}, \varepsilon_{2N+3})$.

There are also many non-parallelizable manifolds that admit almost Kronecker structures (see Appendix \ref{nonpar_sec}).
Before formulating the topological criterion for exsitence of almost Kronecker sturcture, note that from \eqref{kernel_coordinate} it
follows that, for every $q\in M$, the $1$-dimensional kernels of the forms in
the pencil $L\bigl(D(q)^\perp\bigr)$ span an $N+1$-subspace $\mathcal V_1(q)$
of $D(q)$ (in the basis $(e_1, \cdots, e_{2N+1})$ used in
\eqref{kernel_coordinate}, this subspace is   spanned by $\{e_1, \dots,
	e_{N+1}\}$). Furthermore, equation \eqref{kernel_coordinate} demonstrates that
as $[\mu : \lambda]$ varies in $\mathbb{P}^1$, the $1$-dimensional kernels
sweep out a rational normal curve $\mathcal C_q$ of degree $N$ in
$\mathbb{P}\bigl(\mathcal V_1(q)\bigr)$ (the image of the Veronese embedding,
\cite[Lecture1]{HarrisAG}). Consequently, the space $\mathcal V_1(q)$ is canonically identified with the $N$-th symmetric power of the pencil, twisted by a line: the map $\alpha\mapsto L(\alpha)^N$ identifies $\mathrm{Sym}^N\bigl(D^\perp(q)\bigr)$ with $\mathcal V_1(q)\otimes\Lambda^{2N+1}D(q)^*$ (see \eqref{volume_ker}). This leads to the following Proposition (below
$\mathrm{Sym}^i$ (resp., $\Lambda ^j$) denotes the symmetric (skew-symmetric)
tensor powers of the corresponding vector bundles, and $\det D:=\Lambda^{2N+1}D$):

\begin{proposition}
	\label{prop:formal_solution}
	Suppose $N \geq 1$.
	A  $(2N+3)$-dimensional manifold $M$ admits a formal solution $(D, L)$ to the differential relation $\mathfrak R$ if and only if there exist a rank $2$ vector bundle $Q$ and a real line bundle $A$ over $M$ such that the tangent bundle $TM$ is isomorphic to the direct sum
	\begin{equation}
		\label{TM_decomp}
		TM \cong \bigl(A\otimes\mathrm{Sym}^N(Q^*)\bigr) \oplus \bigl( A^*\otimes\mathrm{Sym}^{N-1}(Q) \otimes \Lambda^2 Q \bigr)\oplus Q.
	\end{equation}
	For a given formal solution $(D,L)$ one can take $Q= TM/D$ (equivalently,  $Q^*\cong D^\perp$) and $A=\det D$. In particular, $M$ admits a formal solution with orientable $D$ if and only if \eqref{TM_decomp} holds with $A$ trivial.
\end{proposition}

The proposition is proven in Appendix \ref{top_sec}. The line bundle $A$ cannot be dropped in general: in Appendix \ref{nonpar_sec} we give a closed $7$-dimensional manifold for which \eqref{TM_decomp} holds with a nontrivial $A$ but fails for every rank $2$ bundle $Q$ when $A$ is trivial; see Remark \ref{rem:twist} for the reason behind the appearance of $A$. As a direct consequence of Theorem \ref{main_theorem} and Proposition \ref{prop:formal_solution} we get the following

	\begin{corollary}
		\label{existence}
		Suppose $N \geq 1$.
		A  $(2N+3)$-dimensional manifold $M$ admits a rank $(2N+1)$  distribution satisfying the relation $\mathfrak R$ if and only if the tangent bundle of $M$ satisfies \eqref{TM_decomp} for some rank $2$ bundle $Q$ and some line bundle $A$ over $M$.
	\end{corollary}
The global decomposition of the tangent bundle given in \eqref{TM_decomp}
is a strong topological restriction on $M$.  For example, among the spheres
$S^{2N+3}$ with $N\ge1$, the only one satisfying
\eqref{TM_decomp}, and therefore the only one admitting an almost
Kronecker structure and, by Corollary \ref{existence}, a rank
$2N+1$ distribution satisfying our relation $\mathfrak R$, is $S^{7}$, which is
parallelizable.

Indeed, suppose that $S^{2N+3}$ satisfies \eqref{TM_decomp} for some
rank two bundle $Q$ and line bundle $A$.  Since $2N+3\ge5$, the sphere $S^{2N+3}$ is simply connected
and $H^{2}(S^{2N+3};\mathbb Z)=0$; hence $A$ is trivial, $Q$ is orientable and its Euler class
vanishes.  Oriented plane bundles over a CW complex $X$ are classified by their
Euler class in $H^{2}(X;\mathbb Z)$, because they are exactly the underlying real
bundles of complex line bundles and $BSO(2)=BU(1)$ is an Eilenberg--MacLane space
$K(\mathbb Z,2)$; equivalently, over a sphere $S^{m}$ with $m\ge3$ they are
classified by $\pi_{m-1}\bigl(SO(2)\bigr)=\pi_{m-1}(S^{1})=0$ (see, e.g.,
\cite{Hatcher-VBKT}, Proposition 1.14).  Therefore $Q$ is trivial, every summand
on the right-hand side of \eqref{TM_decomp} is trivial, and
$TS^{2N+3}$ is a trivial bundle of rank $2N+3$, i.e. the sphere is
parallelizable.  By the theorem of Bott--Milnor \cite{BottMilnor} and Kervaire
\cite{Kervaire}, or equivalently by Adams' celebrated theorem on the Hopf
invariant one \cite{Adams}, the only parallelizable spheres are $S^{1}$, $S^{3}$
and $S^{7}$, and only the last of them is of dimension $2N+3$ with $N\ge1$; it
corresponds to $N=2$.  Conversely, being parallelizable, $S^{7}$ does satisfy
\eqref{TM_decomp}, with $Q$ and $A$ trivial.

The same argument applies verbatim to every $(2N+3)$-dimensional manifold $M$
with $H^{1}(M;\mathbb Z/2)=0$ and $H^{2}(M;\mathbb Z)=0$ (the first condition forces $A$ to be trivial and $Q$ to be orientable, the second forces $Q$ to be trivial), and in particular to
every $2$-connected one: such a manifold admits an almost Kronecker structure if
and only if it is parallelizable.  For general $M$ this is far from being the
case: in Appendix \ref{nonpar_sec} we give three non-parallelizable
examples on which the desired distribution exists, namely
$\mathbb{RP}^{5}\times T^{4}$ (closed orientable, of dimension $9$, i.e. $N=3$),
$\mathbb{RP}^{2}\times T^{3}$ (closed nonorientable, of dimension $5$, i.e.
$N=1$), both with $A$ trivial, and $\mathbb{RP}^{4}\times T^{3}$ (closed nonorientable, of dimension $7$, i.e. $N=2$), for which $A$ is necessarily nontrivial.

    The paper is organized as follows. In Section \ref{ample_sec}, we recall the definition of ampleness and Gromov's convex integration theorem. Our main theorem is proved in Section \ref{proof-of-main-theorem}. More specifically, in Section \ref{preliminaries}, we reduce the question of ampleness to another bundle on which explicit calculations can be made. In Section \ref{kw-normal-formal}, we review and restate the Kronecker--Weierstrass normal form in the language of two-forms. We show (Proposition \ref{2cases_prop}) that in order for principal subspaces to have a nontrivial intersection with the relation, the corresponding reduced pencil must consist of either regular indecomposable blocks only (with distinct elementary divisors) or precisely two singular indecomposable blocks.

Ampleness in the regular case is proved in Section \ref{ampleness-regular} by showing that the intersection of the defining differential relation with the principal directions is the complement of a thin (i.e., codimension 2) set. The singular case is more involved and interesting: in Section \ref{ampleness-singular}, we reduce the question to the ampleness of the complement of the zero level set of the resultant of two binary forms of given degrees; in Section \ref{sec:components}, we describe the connected components of the complement of this resultant's zero level set; and in Section \ref{sec:ampleness-resultant}, we conclude the proof of ampleness in the singular case. In both cases the key step is an elementary description of the kernel of a skew-symmetric form in terms of its restriction to a hyperplane (Lemma \ref{lem:kernel-extension}).  Appendix \ref{top_sec} contains the proof of Proposition \ref{prop:formal_solution}, and Appendix \ref{nonpar_sec} the examples mentioned above. In Appendix \ref{even_dim_sec}, for completeness, we briefly discuss corank $2$ distributions on even-dimensional manifolds. There the regularity of the associated pencil defines a natural open relation which is ample, so that the $h$-principle holds for it, whereas the finer open relations prescribing the number of real roots of the Pfaffian of the pencil are not ample; for them plain convex integration is insufficient and one would have to resort to the convex integration with avoidance of \cite{MAdP2021}, as was done there for hyperbolic $(4,6)$ distributions.


	\section{Ampleness and convex integration}
	\label{ample_sec}
	In this section, we briefly describe the condition of ampleness of a differential relation. We restrict ourselves to the differential relations on a $1$-jet bundle $X^{(1)}$.
	The material of this section is standard; the details can be found in \cite[Part 4]{EliashbergMishachev2002}.

	Let $V = \ker d\pi$ be the vertical distribution on $X$, consisting of the tangent spaces to the fibers of the bundle $\pi \colon X \to M$. Consider the first jet bundle $X^{(1)}$ as a bundle over $X$.
	The fiber $X^{(1)}|_x$ of this bundle over a point $x \in X$ can be identified with $\Hom(T_{\pi(x)}M, V_x)$,
	\begin{equation}
		\label{identification}
		X^{(1)}|_x \cong \Hom(T_{\pi(x)}M, V_x),
	\end{equation}
	after a choice of a global connection on $X$, i.e., a choice of a splitting $TX = H \oplus V$, or, equivalently, of a projection $\mathbb{A} \colon TX \to V$. The identification \eqref{identification} is described explicitly as follows:
$$J^1 f \bigl(\pi(x)\bigr)\mapsto \mathbb{A}_x \circ df_{\pi(x)},$$
for every local section $f$ of the bundle $X$ with $f\bigl(\pi(x)\bigr)=x$.


Let $x \in X$, $\tau \subseteq T_{\pi(x)} M$ be a hyperplane, and $\ell \colon \tau \to V_x$
a linear map.
The hyperplane $\tau$ is called a \textit{principal direction} at $\pi(x)$.
The \textit{principal subspace} with respect to $\tau$ and $\ell$ is defined to be
\begin{equation}
	\label{principal_subspace}
	P_\tau^\ell = \{ L \in \Hom(T_{\pi(x)} M, V_x) \colon L|_\tau = \ell \} \subseteq X^{(1)}|_x,
\end{equation}
where the inclusion is understood through the identification \eqref{identification}. Obviously, $P_\tau^\ell$ is an affine subspace of $\Hom(T_{\pi(x)} M, V_x)$.
\begin{remark}
	The map $\ell$ is just a coordinate-free way of prescribing first-order derivatives along
	$(n-1)$-directions.
	In local coordinates, $\ell$ can be understood as follows: fix coordinates
	$x_1, \ldots, x_n$ of $M$ such that $\tau = \{ dx_n = 0 \}$.
	Then, $\ell$ prescribes the first-order derivatives along
	$\partial_{x_1}, \ldots, \partial_{x_{n-1}}$.
\end{remark}
\begin{definition}
	A differential relation $\mathcal{R}$ is called \textit{ample} if for every $x\in X$, principal direction $\tau$ at $\pi(x)$, and linear map  $\ell \colon \tau \to V_x$, its  intersection $ \mathcal{R} \cap P_\tau^\ell$ with the corresponding principal subspace $P_\tau^\ell$
	is an ample subset of $P_\tau^\ell$, i.e. if either
	$ \mathcal{R} \cap P_\tau^\ell = \emptyset$ or the convex hull of every
	connected component of $\mathcal{R}\cap P_\tau^\ell  $ equals $P_\tau^\ell$.
\end{definition}

The following  theorem of Gromov connects ampleness with the $h$-principle:
\begin{theorem}[Gromov 1973, \cite{Gromov1973}]
	\label{Gromov_ample_thm}
	Let $\mathcal{R} \subseteq X^{(1)}$ be an open ample differential relation.
	Then $\mathcal R$ satisfies the multi-parametric $C^0$-close $h$-principle.
\end{theorem}

\section{Proof of Theorem \ref{main_theorem}}
\label{proof-of-main-theorem}

\subsection{Preliminaries}
\label{preliminaries}
As already mentioned before, the proof consists of checking ampleness of the differential relation $\mathfrak R$ so that the result will be the direct consequence of Theorem \ref{Gromov_ample_thm}.
In fact, we will work with the relation $\Pi(\mathfrak{R})$ on the intermediate bundle $\widetilde{X}$ instead of the relation $\mathfrak{R}$ on $X^{(1)}$. 
We say that the relation $\Pi(\mathfrak R)$ is ample, if for every principal subspace $P_\tau^\ell$ as in \eqref{principal_subspace}, the set $\Pi(\mathfrak{R}) \cap \Pi(P_\tau^\ell)$ is ample in $\Pi(P_\tau^\ell)$, i.e., either
\(
\Pi(\mathfrak{R}) \cap \Pi(P_\tau^\ell) = \emptyset
\)
or the convex hull of every connected component of $\Pi(\mathfrak{R}) \cap \Pi(P_\tau^\ell)$ equals $\Pi(P_\tau^\ell)$.
Taking into account the identifications \eqref{identification} and \eqref{tildefiber} of the fibers of the bundles $X^{(1)} \to X$ and $\widetilde{X} \to X$, respectively, with vector spaces, as well as the definition of the projection $\Pi$ (see the paragraph after \eqref{tildefiber}) we have the following
\begin{proposition}
\label{Pi_surjective}
The projection $\Pi:X^{(1)}\to \widetilde X$  is fiberwise linear and surjective. Consequently, the relation $\mathfrak{R}$ is ample if $\Pi(\mathfrak{R})$ is ample. 
\end{proposition}
\begin{proof}
The map $\Pi$ admits the following explicit description as an
antisymmetrization, which in particular shows that it is well defined
and surjective.

Set $n:=\dim M=2N+3$ and fix local coordinates $(x^1,\dots,x^n)$
centred at $q$.  A local corank-two distribution $\widehat D$ near $q$
can be written as $\widehat D=\ker\alpha_1\cap\ker\alpha_2$ for a pair
of pointwise linearly independent $1$-forms
$\alpha_k=\sum_{j=1}^n a_{kj}\,dx^j$, $k=1,2$, and we let
\begin{equation}\label{eq:jacobians}
  (J_k)_{ij}:=\frac{\partial a_{kj}}{\partial x^i}(q),
  \qquad k=1,2,
\end{equation}
be the matrices of first derivatives of their coefficients at $q$ (the
transposed Jacobi matrices, the first index being the direction of
differentiation).  The point $x\in X$ determined by $\widehat D(q)$ is
$D^\perp=\operatorname{span}\{\alpha_1(q),\alpha_2(q)\}$, and the pair
$(J_1,J_2)$ determines the $1$-jet $y=J^1\widehat D(q)$ over $x$,
although not bijectively: two pairs of defining forms for the same
distribution differ by $\widetilde\alpha_k=\sum_l c_{kl}\alpha_l$ with
$\bigl(c_{kl}\bigr)$ a $\mathrm{GL}_2(\mathbb{R})$-valued function, and if
$\bigl(c_{kl}(q)\bigr)$ is the identity (so that the basis of $D^\perp$
is unchanged) then
\begin{equation}\label{eq:gauge}
  \widetilde J_k=J_k+\sum_{l=1,2} dc_{kl}(q)\otimes\alpha_l(q).
\end{equation}
Thus the fibre $X^{(1)}|_x$ is identified with the quotient of
$\mathrm{Mat}_n(\mathbb{R})\oplus\mathrm{Mat}_n(\mathbb{R})$ by the subspace
$\bigl(T^*_qM\otimes D^\perp\bigr)^{\oplus2}$ of the terms appearing in
\eqref{eq:gauge}; the dimensions agree with \eqref{identification}, since
$2n^2-2\cdot 2n=2n(n-2)=\dim\operatorname{Hom}\bigl(T_qM,V_x\bigr)$.

By \eqref{eq:jacobians},
$\bigl(J_k-J_k^{T}\bigr)_{ij}
 =\partial_i a_{kj}(q)-\partial_j a_{ki}(q)
 =d\alpha_k(q)\bigl(\partial_{x^i},\partial_{x^j}\bigr)$,
i.e.\ $J_k-J_k^{T}$ is precisely the matrix of the $2$-form
$d\alpha_k(q)$ on $T_qM$.  Using the basis $\bigl(\alpha_1(q),\alpha_2(q)\bigr)$
of $D^\perp$ to identify the fibre
$\widetilde X_D=\operatorname{Hom}\bigl(D^\perp,\bigwedge^2 D^*\bigr)$ of
\eqref{tildefiber} with pairs of skew-symmetric forms on $D=\widehat D(q)$,
the definition of the dual Levi map, $L_q(\alpha_k)=d\alpha_k|_{D}$,
therefore reads
\begin{equation}\label{eq:Pi-antisym}
  \Pi(y)=\Bigl(\bigl(J_1-J_1^{T}\bigr)\big|_{D},\ \bigl(J_2-J_2^{T}\bigr)\big|_{D}\Bigr):
\end{equation}
\emph{the map $\Pi$ is antisymmetrization followed by restriction to
$D$.}

The restriction to $D$ is exactly what makes \eqref{eq:Pi-antisym}
independent of the choices: each ambiguity term in \eqref{eq:gauge}
contributes $v\otimes\alpha_l-\alpha_l\otimes v$ with $v=dc_{kl}(q)$, and
$\bigl(v\otimes\alpha_l-\alpha_l\otimes v\bigr)(X,Y)
 =v(X)\alpha_l(Y)-v(Y)\alpha_l(X)=0$ for $X,Y\in D$, since
$\alpha_l|_{D}=0$.  (Note also that the individual matrices $J_k$ depend
on the choice of coordinates, whereas their antisymmetrizations
$J_k-J_k^{T}$, being the matrices of the $2$-forms $d\alpha_k(q)$, do
not.)  For a general change of defining forms, with
$\bigl(c_{kl}(q)\bigr)\in\mathrm{GL}_2(\mathbb{R})$ arbitrary, both sides of
\eqref{eq:Pi-antisym} transform by the same matrix
$\bigl(c_{kl}(q)\bigr)$, so $\Pi$ does not depend on the choice of the
basis of $D^\perp$ either.

Fiberwise linearity of $\Pi$ is immediate from the description, since
        $J\mapsto\bigl(J-J^{T}\bigr)\big|_{D}$ is linear and \eqref{eq:gauge} identifies
        the fibre as a quotient by a linear subspace.
Besides, \eqref{eq:Pi-antisym} makes the surjectivity of $\Pi$ evident:
given $\Lambda\in\operatorname{Hom}\bigl(D^\perp,\bigwedge^2D^*\bigr)$,
extend each of the two skew-symmetric forms $\Lambda(\alpha_k(q))$ on
$D$ to a skew-symmetric matrix $A_k$ on $T_qM$ and take
$J_k:=\tfrac12A_k$, so that $J_k-J_k^{T}=A_k$; any local distribution
whose defining forms have these first derivatives at $q$ realizes
$\Lambda$ as its dual Levi map at $q$. Finally, the  surjective linear map $\Pi$
        carries the affine subspace $\Pt$ onto $\Pi(\Pt)$ with affine fibres,
        so $\pi_0\bigl(\Pi^{-1}(S)\bigr)\cong\pi_0(S)$ and
        $\operatorname{conv}\bigl(\Pi^{-1}(C)\bigr)
         =\Pi^{-1}\bigl(\operatorname{conv}C\bigr)$.
\end{proof}

 By Proposition \ref{Pi_surjective} the ampleness of $\Pi(\mathfrak R)$ implies our main Theorem \ref{main_theorem}), so we will focus on showing the former.
 
Fix  a point $x$ in the bundle $X$, i.e. a point $q=\pi(x)\in M$ and a codimension $2$ subspace $D$ of $T_qM$. Also, fix a principal direction $\tau$ (i.e., a hyperplane in $T_qM$) and a linear map $\ell:\tau\to V_x$, where, as before,  $V$ is the vertical distribution of the bundle $X\to M$. Let $r_\tau: \bigwedge^2 D^*\to  \bigwedge^2 (\tau\cap D)^*$ be the operator of restriction of  skew-symmetric $2$-forms on $D$ to $\tau\cap D$,
\begin{equation}
	\label{r_tau}
	r_\tau(\omega):=\omega|_{\tau\cap D}, \quad \forall \omega\in \textstyle{\bigwedge^2} D^*\cong (\textstyle{\bigwedge^2} D)^*\end{equation}
Further, given a point $L$ in the fiber of $\widetilde X$ over $x$, i.e. an element of  $\mathrm{Hom}\bigl(D^\perp, \bigwedge^2 D^*\bigr)$, define
$L_\tau\in \mathrm{Hom}\bigl(D^\perp, \bigwedge^2 (\tau\cap D)^*\bigr)$ by
\begin{equation}
	\label{L_tau}
	L_\tau(\alpha):=r_\tau\circ L(\alpha), \quad  \forall\alpha\in D^\perp)
\end{equation}
\begin{remark}
	\label{the_same_rem}
	By construction, for every $y\in P_\tau^\ell$, $\Pi(y)_\tau$ is the same element of $\mathrm{Hom}\bigl(D^\perp, \bigwedge^2 (D\cap \tau)^*\bigr)$. Consequently, $\mathrm{Im}(\Pi(y)_\tau)$ is the same subspace of $\bigwedge^2 (D\cap \tau)^*$ for every $y\in P_\tau^\ell$.
\end{remark}
First, assume that  $D\subset \tau$. Then $\Pi(y)=\Pi(y)_\tau$ for all $y\in P_\tau^\ell$. Therefore, $\Pi( P_\tau^\ell)$ consists of exactly  one element of $\mathrm{Hom}\bigl(D^\perp, \bigwedge^2 D^*\bigr)$, which is either in $\Pi(\mathfrak R)$ or not. Therefore, in this case
$\Pi(\mathfrak R)\cap \Pi( P_\tau^\ell)$ is trivially ample in $\Pi(P_\tau^\ell)$.

Now assume that $D\not\subset\tau$. Denote by $Z_\tau$ the subspace in $\bigwedge^2 D^*$ consisting of all skew-symmetric $2$-forms on $D$ whose restriction to $\tau\cap D$ is identically zero.
\begin{equation}
	\label{Z_tau}
	Z_\tau=\{\alpha\in \bigwedge^2 D^*: \alpha|_{\tau\cap D}=0\}
\end{equation}

Further, fixing a line $e$ in $D$ transverse to $\tau\cap D$, i.e. such that $D=(\tau\cap D)\oplus e$, let $K_e$ denotes the subspace in $\bigwedge^2 D^*$ consisting of all skew-symmetric $2$-forms on $D$ whose kernel contains $e$. Then we have the following splitting
\begin{equation}
	\label{form_splitting}
	\textstyle{\bigwedge^2} D^*=K_e\oplus Z_\tau.
\end{equation}
Moreover, if  $\varepsilon \in D^*$ such that
\begin{equation}
	\label{varepsilon}
	\ker\, \varepsilon =\tau\cap D
\end{equation}
then
\begin{equation}
	\label{K_tau}
	Z_{\tau}=\mathrm{Im} (\varepsilon \,\wedge)=\{\varepsilon \wedge \beta: \beta \in D^*\},
\end{equation}
where $\varepsilon \wedge: D^*\to \wedge^2 D^*$ is defined by $\beta\mapsto \varepsilon\wedge \beta$.

Let $\mathrm{pr}_{K_e}$ denote the projection to $K_e$ with respect to the splitting \eqref{form_splitting}. Note that there is a natural identification
\begin{equation}
	\label{eperp}
	(\tau\cap D)^*\cong D^*\cap (e)^\perp,
\end{equation}
where, similarly to the previous notations,  $e^\perp$ denotes the space of elements of $D^*$ annihilating $e$. Taking into account this identification and splitting \eqref{form_splitting} there exists the unique map $\beta \in \mathrm{Hom}\bigl(\wedge^2 D^*, (\tau\cap D)^*\bigr)$ such that
for any $\omega \in \bigwedge^2 D^*$ we have
\begin{equation}
	\omega=\mathrm{pr}_{K_e} (\omega)+\beta(\omega)\wedge \epsilon.
	\label{omegasplitting}
\end{equation}
Besides, the operator $r_\tau$  as in \eqref{r_tau}, restricted to $K_e$, defines the bijection between $K_e$ and  $ \bigwedge^2 (\tau\cap D)^*$,
\begin{equation}
	\label{Ke_id}
	K_e\stackrel{r_\tau}{\cong}\textstyle{\bigwedge^2} (\tau\cap D)^*.
\end{equation}
Using this identification and Remark \ref{the_same_rem} we can conclude that
for every $y \in P_\tau^\ell$, the map
\begin{equation}
	\label{Ltauell0}
	\alpha\mapsto \mathrm{pr}_{K_e}\bigl( \Pi(y)(\alpha)\bigr), \quad \alpha \in D^\perp
\end{equation}
is the same element of $\mathrm{Hom}(D^\perp, K_e)$, which will be called the \emph{reduced pencil associated with} $P_\tau^\ell$ and will be  denoted by $L_\tau^\ell$,
\begin{equation}
	\label{L_tau_ell}
L_\tau^\ell:=\mathrm{pr}_{K_e}\circ \Pi(y),  \quad y\in P_\tau^\ell.
\end{equation}

Finally, from the latter statement and decomposition \eqref{omegasplitting} it follows that the map
$$B: \Pi(P_\tau^\ell) \to \mathrm{Hom}(D^\perp, (\tau\cap D)^*)$$ define by

\begin{equation}
	\label{B_map}
	L\in \Pi(P_\tau^\ell)\stackrel{B}{\longmapsto} \Bigl(\alpha\in D^\perp\mapsto \beta(L\bigl(\alpha)\bigr)\Bigl),
\end{equation}
where $\beta$ is as in \eqref{omegasplitting}, is a linear bijection.
Therefore, ampleness of $\Pi(\mathfrak R)\cap \Pi(P_\tau^\ell)$ in $\Pi(P_\tau^\ell)$ will follow from ampleness of the image of $\Pi(\mathfrak R)\cap \Pi(P_\tau^\ell)$ under $B$ in
$\mathrm{Hom}(D^\perp, (\tau\cap D)^*)$.

Fix a basis $(p_1, p_2)$ in $D^\perp$. Recalling that $\Pi(y)\in \mathrm{Hom}\bigl(D^\perp, \bigwedge^2 D^*\bigr)$,
let
\begin{equation}
\label{translation_to_old_1}
	\omega_i:=\Pi(y)(p_i), \quad \overline\omega_i:=\mathrm{pr}_{K_e} (\omega_i), \quad \beta_i:=\beta(\omega_i), \quad i=1,2,
\end{equation}
so that the decomposition \eqref{omegasplitting} for forms $\omega_i$ reads
\begin{equation}
	\label{translation_to_old_2}
	\omega_i=\overline \omega_i+\beta_i\wedge \epsilon, \quad i=1,2.
\end{equation}
By construction,
\begin{equation}
	\label{baromega}
	\mathrm{Im}\, L_\tau^\ell=\mathrm{span}\,\{\overline \omega_1, \overline\omega_2 \}.
\end{equation}

As before, we assume that $\dim D=2N+1$. 

From \eqref{translation_to_old_2} and $(\beta_i\wedge\epsilon)^2=0$ we obtain

\begin{equation}\label{eq:expansion}
\omega_1^{\ell}\wedge\omega_2^{N-\ell}
  =\overline{\omega}_1^{\ell}\wedge\overline{\omega}_2^{N-\ell}
  +\ell\,\overline{\omega}_1^{\ell-1}\wedge\overline{\omega}_2^{N-\ell}\wedge\beta_1\wedge\varepsilon
  +(N-\ell)\,\overline{\omega}_1^{\ell}\wedge\overline{\omega}_2^{N-\ell-1}\wedge\beta_2\wedge\varepsilon,
\end{equation}
for $0\le\ell\le N$, where the terms with a negative exponent are multiplied by
zero and are therefore absent.  We will need the following lemma.

\begin{lemma}\label{lem:span-at-most-two}
Assume that $\overline{\omega}_2=c\,\overline{\omega}_1$ for some $c\ne 0$, and set
\begin{align}
\label{eq:v}
  v&:=\overline{\omega}_1^{N}+\frac{N}{c}\,\overline{\omega}_1^{N-1}\wedge\beta_2\wedge\varepsilon,\\
\label{eq:w}
  w&:=\overline{\omega}_1^{N-1}\wedge\beta_1\wedge\varepsilon
     -\frac{1}{c}\,\overline{\omega}_1^{N-1}\wedge\beta_2\wedge\varepsilon .
\end{align}
Then
\begin{equation}\label{eq:power-formula}
  \omega_1^{\ell}\wedge\omega_2^{N-\ell}=c^{\,N-\ell}\,(v+\ell\,w),
  \qquad 0\le\ell\le N .
\end{equation}
In particular, the span of $\{\omega_1^{\ell}\wedge\omega_2^{N-\ell}\}_{\ell=0}^{N}$
is at most two-dimensional.
\end{lemma}

\begin{proof}
Substituting $\overline{\omega}_2=c\overline{\omega}_1$ into \eqref{eq:expansion} gives
\[
  \omega_1^{\ell}\wedge\omega_2^{N-\ell}
  =c^{\,N-\ell}\Bigl(\overline{\omega}_1^{N}
   +\ell\,\overline{\omega}_1^{N-1}\wedge\beta_1\wedge\varepsilon
   +(N-\ell)\,c^{-1}\overline{\omega}_1^{N-1}\wedge\beta_2\wedge\varepsilon\Bigr),
\]
and the expression in brackets equals $v+\ell w$ by \eqref{eq:v}--\eqref{eq:w}.
\end{proof}

\begin{proposition}\label{ample_dim_1}
Assume $\dim\bigl(\operatorname{Im}L^{\ell}_{\tau}\bigr)<2$.  Then:
\begin{enumerate}
  \item if $N>1$, then $\Pi(\mathfrak{R})\cap\Pi(P^{\ell}_{\tau})=\emptyset$;
  \item if $N=1$ and $\dim\bigl(\operatorname{Im}L^{\ell}_{\tau}\bigr)=1$, then
        $\Pi(\mathfrak{R})\cap\Pi(P^{\ell}_{\tau})$ is the complement of a thin
        (i.e.\ codimension two) affine subspace of $\Pi(P^{\ell}_{\tau})$;
  \item if $N=1$ and $L^{\ell}_{\tau}=0$, then
        $\Pi(\mathfrak{R})\cap\Pi(P^{\ell}_{\tau})$ is the complement of the zero
        locus of a nondegenerate quadratic form of signature $(2,2)$, and it has
        exactly two connected components.
\end{enumerate}
In all three cases $\Pi(\mathfrak{R})\cap\Pi(P^{\ell}_{\tau})$ is an ample subset
of $\Pi(P^{\ell}_{\tau})$.
\end{proposition}

\begin{proof}
We use the characterization of generic orbits in
$\mathrm{Gr}_2\bigl(\bigwedge^{2}D^{*}\bigr)$ given by item (4) of
Proposition 1.3.  Note first that $\Pi(y)\in\Pi(\mathfrak{R})$ if and only if the
$N+1$ forms $\{\omega_1^{\ell}\wedge\omega_2^{N-\ell}\}_{\ell=0}^{N}$ are linearly
independent: this condition already forces $\omega_1,\omega_2$ to be linearly
independent, and hence to be a basis of the pencil $\operatorname{Im}\Pi(y)$, to
which item (4) of Proposition 1.3 then applies.

Since $\dim\bigl(\operatorname{Im}L^{\ell}_{\tau}\bigr)\le1$, the forms
$\overline{\omega}_1,\overline{\omega}_2$ are proportional, so after a change of the basis $(p_1,p_2)$
of $D^{\perp}$ we may assume that
\[
  \overline{\omega}_2=c\,\overline{\omega}_1\ \text{ with }c\ne0,
  \qquad\text{where}\qquad
  \begin{cases}
    \overline{\omega}_1\ne0, & \text{if }\dim\bigl(\operatorname{Im}L^{\ell}_{\tau}\bigr)=1,\\[2pt]
    \overline{\omega}_1=\overline{\omega}_2=0, & \text{if }\dim\bigl(\operatorname{Im}L^{\ell}_{\tau}\bigr)=0 .
  \end{cases}
\]
Such a change of basis induces a linear automorphism of
$\operatorname{Hom}\bigl(D^{\perp},(\tau\cap D)^{*}\bigr)$ and therefore affects
neither connectedness nor convex hulls; it is thus harmless for the study of
ampleness.

(1)  By Lemma~\ref{lem:span-at-most-two} the span of
$\{\omega_1^{\ell}\wedge\omega_2^{N-\ell}\}_{\ell=0}^{N}$ is at most
two-dimensional, whereas item (4) of Proposition 1.3 requires these $N+1\ge3$
forms to be linearly independent.  Hence no point of $P^{\ell}_{\tau}$ satisfies
the relation when $N>1$, which proves item (1).  An empty set is ample by
definition.

(2)  Let $N=1$ and $\overline{\omega}_1\ne0$.  By \eqref{eq:power-formula} the two forms in
question are $\omega_2=c\,v$ and $\omega_1=v+w$, so they are linearly dependent
if and only if $v$ and $w$ are.  By \eqref{eq:v}--\eqref{eq:w} with $N=1$ we have
$v=\overline{\omega}_1+\frac1c\beta_2\wedge\varepsilon$ and
$w=\bigl(\beta_1-\frac1c\beta_2\bigr)\wedge\varepsilon$; since $w\in Z_{\tau}$
while $v\notin Z_{\tau}$ (its $K_e$-component is $\overline{\omega}_1\ne0$), the pair $v,w$
is linearly dependent if and only if $w=0$, that is, if and only if
\[
  \Bigl(\beta_1-\tfrac1c\beta_2\Bigr)\wedge\varepsilon=0,
  \qquad\text{i.e.}\qquad \beta_1=\tfrac1c\,\beta_2 ,
\]
the last equivalence because $\beta_1-\frac1c\beta_2$ lies in
$(\tau\cap D)^{*}\cong D^{*}\cap(e)^{\perp}$, which intersects
$\operatorname{span}\{\varepsilon\}$ trivially.  As
$\dim(\tau\cap D)^{*}=2$, this is a condition of codimension two on
$(\beta_1,\beta_2)\in\operatorname{Hom}\bigl(D^{\perp},(\tau\cap D)^{*}\bigr)$,
which proves item (2).  The complement of an affine subspace of codimension two
is connected and its convex hull is the whole space, so the intersection is
ample.

(3)  Let $N=1$ and $\overline{\omega}_1=\overline{\omega}_2=0$, so that
$\omega_i=\beta_i\wedge\varepsilon$ for $i=1,2$.  Since
$\beta\mapsto\beta\wedge\varepsilon$ is injective on $(\tau\cap D)^{*}$, the
forms $\omega_1,\omega_2$ are linearly independent if and only if $\beta_1$ and
$\beta_2$ are.  Choosing a basis of $(\tau\cap D)^{*}$ and writing
$\beta_i=(\beta_{i1},\beta_{i2})$, this means that
\[
  Q(\beta_1,\beta_2):=\beta_{11}\beta_{22}-\beta_{12}\beta_{21}\neq0 ,
\]
the determinant of the corresponding $2\times2$ matrix.  
The
quadratic form $Q$ is nondegenerate of signature $(2,2)$, which proves the first
assertion of item (3).  Writing $Q=|u|^{2}-|u'|^{2}$ in suitable linear
coordinates $(u,u')\in\mathbb{R}^{2}\times\mathbb{R}^{2}$, the set $\{Q>0\}$
deformation retracts onto $\{|u|=1,\,u'=0\}\cong S^{1}$ and the set $\{Q<0\}$
onto a circle as well; both are therefore connected, and
$\Pi(\mathfrak{R})\cap\Pi(P^{\ell}_{\tau})=\{Q\ne0\}$ has exactly two connected
components.  Note that, in contrast with item (2), this is the complement of a
set of codimension one, and it is not connected.  Its ampleness is nevertheless
classical: the set of linearly independent $k$-tuples of vectors in a vector
space is ample, see \cite[Section 20.4, p.~176]{EliashbergMishachev2002}.
\end{proof}

The following elementary lemma describes how the kernel of a $2$-form changes when the form is extended from a hyperplane by one row and one column, i.e. when passing from $\overline\omega_i$ to $\omega_i$ in \eqref{translation_to_old_2}. It will be used, over $\mathbb R$ and over $\mathbb C$, in both cases of Proposition \ref{2cases_prop} below.

\begin{lemma}
	\label{lem:kernel-extension}
	Let $H$ be a vector space of even dimension $2N$ over $\mathbb F=\mathbb R$ or $\mathbb C$, let $V=H\oplus\mathbb F e$, and let $\varepsilon\in V^*$ be the covector with $\varepsilon|_H=0$ and $\varepsilon(e)=1$. Let $\overline\omega\in\bigwedge^2V^*$ be a $2$-form with $\iota_e\overline\omega=0$, let $\beta\in V^*$ with $\beta(e)=0$, and set
	\[
		\omega:=\overline\omega+\beta\wedge\varepsilon,\qquad K:=\ker\bigl(\overline\omega|_H\bigr)\subset H .
	\]
	Then $\dim\ker\omega\ge\dim K-1$, and moreover:
	\begin{enumerate}
		\item if $K=0$, then $\dim\ker\omega=1$;
		\item if $\dim K=2$, then $\dim\ker\omega=1$ if $\beta|_K\neq0$, and $\dim\ker\omega=3$ if $\beta|_K=0$.
	\end{enumerate}
\end{lemma}

\begin{proof}
	Write a vector of $V$ as $v=h+ce$ with $h\in H$ and $c\in\mathbb F$. Since $\iota_{h+ce}(\beta\wedge\varepsilon)=\beta(h)\,\varepsilon-c\,\beta$ and $\iota_{h+ce}\overline\omega=\iota_h\overline\omega\in H^*$, the vector $v$ lies in $\ker\omega$ if and only if
	\begin{equation}
		\label{eq:kernel-extension}
		\beta(h)=0\qquad\text{and}\qquad\iota_h\overline\omega=c\,\beta\quad\text{in }H^* .
	\end{equation}
	The linear map $h\mapsto\iota_h\overline\omega$ has kernel $K$ and image $K^{\perp}:=\{\xi\in H^*:\xi|_K=0\}$.

	The solutions of \eqref{eq:kernel-extension} with $c=0$ are the vectors $h\in K\cap\ker\beta$; they form a space of dimension $\dim K$ if $\beta|_K=0$ and of dimension $\dim K-1$ otherwise. This already gives $\dim\ker\omega\ge\dim K-1$. Solutions with $c\neq0$ exist if and only if $\beta\in K^{\perp}$, i.e. $\beta|_K=0$. In that case choose $h_0\in H$ with $\iota_{h_0}\overline\omega=\beta$; then $\beta(h_0)=\overline\omega(h_0,h_0)=0$, so $h_0+e\in\ker\omega$, and
	\[
		\ker\omega=(K\cap\ker\beta)\oplus\mathbb F\,(h_0+e)=K\oplus\mathbb F\,(h_0+e).
	\]
	If $K=0$, then $\beta|_K=0$ holds trivially and $\ker\omega=\mathbb F(h_0+e)$ is one-dimensional, which proves (1). If $\dim K=2$ and $\beta|_K\neq0$, there are no solutions with $c\neq0$ and $\ker\omega=K\cap\ker\beta$ is one-dimensional; if $\dim K=2$ and $\beta|_K=0$, then $\ker\omega=K\oplus\mathbb F(h_0+e)$ is three-dimensional. This proves (2).
\end{proof}



\subsection{The Kronecker--Weierstrass normal form}
\label{kw-normal-formal}
Proposition \ref{ample_dim_1} implies that it remains to consider the case when $\dim\, (\mathrm{Im} L_\tau^\ell)=2$.
In this subsection, we recall the theory of Kronecker-Weierstrass normal forms for pairs (and therefore pencils) of skew-symmetric forms, developed by M. Gauger \cite{gauger1973}, based on the classical Kronecker--Weierstrass theory of pencils of matrices \cite[Chapter XII]{Gantmacher}.

Note that while this theory is developed for pencils over algebraically closed fields, the results for pencils over $\mathbb{R}$ can be deduced from those over $\mathbb{C}$ in a standard way. Throughout this subsection, we work with vector spaces over a field $\mathbb{F}$, where $\mathbb{F}$ is either $\mathbb{R}$ or $\mathbb{C}$.

We first recall general terminology. A pair of skew-symmetric forms $(\omega_1, \omega_2)$ (or the pencil $\operatorname{span}(\omega_1, \omega_2)$) on a vector space $V$ is called \emph{decomposable} if there exists a proper splitting $V = V_1 \oplus V_2$ (i.e., neither $V_i$ is zero) such that
\begin{equation}
	\label{indecommp}
	\omega_i(v_1, v_2)=0, \,\, \forall \,i=1,2, \quad  v_1\in V_1, \quad v_2\in V_2,
\end{equation}
and in this case the pair of skew-symmetric forms is said to be the direct sum of the pairs restricted to the corresponding subspaces, $(\omega_1|_{V_1}, \omega_2|_{V_1})$ and $(\omega_1|_{V_2}, \omega_2|_{V_2})$. Otherwise, the pair  $(\omega_1, \omega_2)$ is called \emph{indecomposable}. We use $\oplus$ to denote the direct sum of pairs of skew-symmetric forms or pencils, and the individual summands in this direct sum will be called \emph{blocks}.

A pencil of skew-symmetric forms (as well as any pair generating it) is called regular if it contains a nondegenerate form (hence, a generic form in it is nondegenerate), and singular otherwise (i.e., all forms in the pencil are degenerate).

To describe all indecomposable pairs of skew-symmetric forms, let us introduce the basic building blocks. Given a $d$-dimensional  vector space $V$ over $\mathbb F$ , let $\mathcal E=\{\varepsilon_j\}_{j=1}^d$ be a basis of the dual space $V^*$, if $\mathbb F=\mathbb C$, or of its complexification  $V^{*\mathbb{C}}$ ($\cong (V^\mathbb C)^*$), if $\mathbb F=\mathbb R$.
We define the following  skew-symmetric forms:

\begin{equation}
	\label{indecomp_phi}
	\Phi_{d}^{\mathcal E}
	= \sum_{j=1}^{\lfloor d/2\rfloor}\varepsilon_{j}
	\wedge \varepsilon_{d+1-j}
	\quad
	\Psi_{d}^{\mathcal E}
	= \sum_{j=2}^{\lceil d/2\rceil} \varepsilon_{j}
	\wedge \varepsilon_{d+2-j}
	,
\end{equation}
where $\lfloor x \rfloor$ and $\lceil x \rceil$ denote the floor and ceiling functions, respectively.

Let
\begin{align}
	~ & \label{reg_finite}
	E_d^\mathcal E(a)=(\Phi_d^\mathcal E,a \Phi_{d}^\mathcal E + \Psi_{d}^\mathcal E),
	\quad  a\in \mathbb C,
	\\
	~ & \label{reg_infinite}
	E_d^\mathcal E(\infty)=(\Psi_d^\mathcal E,\Phi_{d}^\mathcal E)
\end{align}
(cf. the latter with equation \eqref{singular block} for $d=2N+1$),
where in the case of real vector space $V$ , if  $a\in \mathbb R\cup\infty$, then  $\mathcal E$
is taken as a basis of $V^*$.


\tb{For simplicity, we first formulate the classification theorem on pairs of skew-symmetric forms over $\mathbb{C}$ via classification of indecomposable forms:
	\begin{theorem}(\cite[Theorem 6.8, p. 309]{gauger1973})
		\label{pencil_complex_class}
		Let $V$ be a $d$-dimensional complex vector space.
		\begin{enumerate}
			\item
			      A pair $(\omega_1, \omega_2)$ of skew-symmetric forms on $V$ is regular and indecomposable
			      if and only if
			      $d$ is even, and $(\omega_1, \omega_2)$ is isomorphic
			      to $E_{d}^\mathcal E(a)$ for some $a \in \mathbb{C} \cup \{\infty\}$ and a basis $\mathcal E$ of $V^*$;
			\item A pair $(\omega_1, \omega_2)$ of skew-symmetric forms on $V$ is singular and indecomposable 
			      if and only if $d$ is odd and $(\omega_1, \omega_2)$ is isomorphic to $E_{d}^\mathcal E(\infty)$ for a basis $\mathcal E$ of $V^*$.
		\end{enumerate}
	\end{theorem}
	Consequently, every pair $(\omega_1, \omega_2)$ of skew-symmetric forms on $V$ is a direct sum of indecomposable pairs from items 1 and 2. The tuple of parameters $a$ and $d$ of the indecomposable blocks appearing in this decomposition (up to permutation of the blocks) uniquely determines the orbit of the pair $(\omega_1, \omega_2)$ under the natural $\mathrm{GL}(V)$-action (we match the parameters $a$'s and $d$'s with the classical terminology of \emph{elementary divisors} and \emph{minimal (or Kronecker) indices} in Remark \ref{terminology_rem} below).}

\tb{In order to obtain the corresponding classification of pairs $(\omega_1,
\omega_2)$ over a real vector space $V$, one complexifies the pair as usual.
That is, each form $\omega_i$ ($i=1,2$) is extended to the complexification
$V^{\mathbb{C}}$ of $V$ in a natural way by $\mathbb{C}$-linearity. Denote such
an extension by $\omega_i^{\mathbb{C}}$. In the decomposition of the pair
$(\omega_1^{\mathbb{C}}, \omega_2^{\mathbb{C}})$ into indecomposable blocks, a
block $E_{d}^{\mathcal{E}}(a)$ with $a \in \mathbb{C} \setminus \mathbb{R}$
must appear together with its complex conjugate
$E_{d}^{\overline{\mathcal{E}}}(\bar{a})$}\footnote{Here,
$\overline{\mathcal{E}} = \{\bar{\varepsilon}_j\}_{j=1}^d$ if $\mathcal{E} =
\{\varepsilon_j\}_{j=1}^d$.} \tb{and the space $W(\mathcal{E}) :=
\mathrm{span}_{\mathbb{C}} \mathcal{E}$ is totally complex in
$(V^{\mathbb{C}})^*$ ($\cong (V^*)^{\mathbb{C}}$), meaning $W(\mathcal{E}) \cap
\overline{W(\mathcal{E})} = 0$. The pair $\bigl(E_{d}^{\mathcal{E}}(a),
E_{d}^{\overline{\mathcal{E}}}(\bar{a})\bigr)$ of indecomposable complex
blocks, which are complex conjugates of each other, produces one indecomposable
real block. This real block is defined on the real space
$\Bigl(\operatorname{Re}\bigl(W(\mathcal{E}) \oplus
\overline{W(\mathcal{E})}\bigr)\Bigr)^*$, which is naturally identified with a
subspace of $V$,}\footnote{Recall that the \emph{real part} or \emph{real
form} of a complex vector space $Z$ equipped with a complex involution
(conjugation) is by definition the fixed-point set of this involution.} \tb{as
the restriction of $E_{d}^{\mathcal{E}}(a) \oplus
E_{d}^{\overline{\mathcal{E}}}(\bar{a})$ to
$\bigl(\operatorname{Re}\bigl(W(\mathcal{E}) \oplus
\overline{W(\mathcal{E})}\bigr)\bigr)^*$. Such an indecomposable block will be
denoted by $\operatorname{Re}\bigl(E_{d}^{\mathcal{E}}(a) \oplus
E_{d}^{\overline{\mathcal{E}}}(\bar{a})\bigr)$.}

\tb{\begin{theorem} [classification of indecomposable pairs over $\mathbb{R}$]
		\label{pencil_real_class}
		Let $V$ be a $d$-dimensional real vector space.
		\begin{enumerate}
			[label=\arabic*]
			\item  A pair $(\omega_1, \omega_2)$ of skew-symmetric forms on $V$ is regular and indecomposable (over $\mathbb{R}$) if and only
			      $d$ is even  and $(\omega_1, \omega_2)$ is isomorphic
			      over $\mathbb{R}$
			      to either
			      \begin{enumerate}
				      \item
				            to $E_{d}^\mathcal E(a)$ for some $a \in \mathbb{R} \cup \{\infty\}$ and a basis  $\mathcal E$ in $V^*$, or   \item
				            $\mathrm {Re}\Bigl(E_{d/2}^\mathcal E(a)\oplus E_{d/2}^{\overline {\mathcal E}}(\bar a)\bigr)$ for some $a\in \mathbb{C} \setminus \mathbb {R}$ and a  tuple of $d/2$ vectors $\mathcal E$ in $(V^*)^{\mathbb C}$ such that $\mathrm{span}_{\mathbb C}\mathcal E\cap \overline{\mathrm{span}_{\mathbb C}\mathcal E}=0$ (note that in this case $d$ is divisible by $4$, each of the two complex blocks having complex dimension $d/2$).
			      \end{enumerate}
			\item A pair $(\omega_1, \omega_2)$ of skew-symmetric forms on $V$ is singular and indecomposable (over $\mathbb{R}$) if and only if $d$ is odd and $(\omega_1, \omega_2)$ is isomorphic over $\mathbb{R}$ to $E_{d}(\infty)$.
		\end{enumerate}
	\end{theorem}}

\tb{Based on the last theorem, the  pair $(\bar\omega_1, \bar\omega_2)$ defined in \eqref{translation_to_old_2} (which is by \eqref{baromega} a basis of $\mathrm{Im}\, L_\tau^\ell$)  decomposes into a direct sum of real indecomposable pairs:
	\begin{equation}
		\label{pair_decomp}
		(\bar\omega_1, \bar\omega_2) = \underbrace{\bigoplus_{i=1}^m E_{2n_i}^{\mathcal E^i}(a_i)
			\oplus \bigoplus_{i=m+1}^p
			\operatorname{Re}
			\bigl(E^{\mathcal E^i}_{2n_i}(a_i) \oplus E_{2n_i}^{\overline{\mathcal E^i}}(\bar a_i)\bigr)}_{\text{regular part}} \oplus \underbrace{\bigoplus_{j=1}^s E_{2r_j+1}^{\mathcal E^{p+j}}(\infty)}_{\text{singular part}},\end{equation}
	where $0\leq m\leq p$,  $a_i\in \mathbb R\cup \{\infty\}$ for $1\leq i\leq m$,  $\mathrm {Im}\, a_i>0$ for $m+1\leq i\leq p $,  $n_i > 0$, $r_j \ge 0$, and $\mathcal E^i$, $1\leq i\leq p+s$ are bases of the dual subspaces of the corresponding indecomposable (over $\mathbb C$) blocks with the properties as in Theorem \ref{pencil_real_class}. Note that since forms  $\bar\omega_i$ operate on an even-dimensional space, the number $s$  of singular indecomposable blocks in \eqref{pair_decomp} is   even.}

Every summand of the regular/singular part of the decomposition \eqref{pair_decomp} will be called a regular/singular indecomposable block of the pair $(\bar\omega_1, \bar\omega_2)$ (or of the corresponding pencil $\mathrm{Im}\, L_\tau^\ell$). A regular indecomposable block will be referred to as \emph{real regular indecomposable} or \emph{complex regular indecomposable} depending on whether $a_i \in \mathbb{R} \cup \{\infty\}$ or $\mathrm{Im}\, a_i > 0$.

\begin{remark}
	\label{terminology_rem}
	To relate the decomposition \eqref{pair_decomp} to the classical terminology, note that:
	\begin{enumerate}
		\item The tuple $(r_j)_{j=1}^s$, ordered in nondecreasing order, is called the tuple of minimal (or Kronecker) indices of the pair $(\bar\omega_1, \bar\omega_2)$ (or of the corresponding pencil $\mathrm{Im}\, L_\tau^\ell$);
		\item If elements of $\mathrm{Im}\, L_\tau^\ell$ ($=\mathrm{span}(\bar\omega_1, \bar\omega_2)$) are denoted by $\mu \bar \omega_1 + \lambda\bar\omega_2$, then a real regular block $E_{2n_i}^{\mathcal E^i}(a_i)$ corresponds to the pair of equal real elementary divisors $(\mu+a_i\lambda)^{n_i}, (\mu+a_i\lambda)^{n_i}$ if $a_i \in \mathbb{R}$, and $\lambda^{n_i}, \lambda^{n_i}$ if $a_i = \infty$. A complex regular block $\operatorname{Re}
			      \bigl(E^{\mathcal E^i}_{2n_i}(a_i) \oplus E_{2n_i}^{\overline{\mathcal E^i}}(\bar a_i)\bigr)$ corresponds to the complex conjugate pair of elementary divisors $(\mu+a_i\lambda)^{n_i}, (\mu+\bar{a}_i\lambda)^{n_i}$ (each appearing twice).
	\end{enumerate}

	For example, the pair $(\omega_1, \omega_2)$ (or the pencil $\mathrm{Im}\,\Pi(y)$) as in \eqref{translation_to_old_2}
	belongs to the generic orbits (or, equivalently, $\Pi(y)\in \Pi(\mathfrak R)$) if and only if its tuple of minimal indices consists of only one element, and the regular part is void (cf. item (3) of Proposition  \ref{open_orbit_char_prop}).
	$\Box$.
\end{remark}

\begin{remark}
	Finally, note that since we consider not just a pair of forms $(\bar \omega_1, \bar\omega_2)$ but the pencil $\mathrm{Im}\, L_\tau^\ell$ generated by them, the parameters $a_i$ defining the elementary divisors are no longer $\operatorname{GL}(V)$-invariants of the pencil. We have the additional freedom to choose a real basis for the pencil, which induces the action of real projective transformations on these parameters. Therefore, we can always choose a basis such that the pencil contains no infinite elementary divisors (i.e., all $a_i$ are finite). We will make this assumption throughout the remainder of the paper.
\end{remark}

The next proposition gives the restriction on the decomposition \eqref{pair_decomp} of the pencil $\mathrm{Im}\, L_\tau^\ell$ under which the pencil $\Pi(y)$ belongs to $\Pi(\mathfrak R)$.
	\begin{proposition}
    \label{2cases_prop}
		Assume that $y\in P_\tau^\ell$. If $\Pi(y) \in \Pi(\mathfrak{R})$, then the decomposition \eqref{pair_decomp} for the reduced pencil $\mathrm{Im}\, L_\tau^\ell$ must either:
		\begin{enumerate}
			\item consist entirely of a regular part in which all $a_i$ are pairwise distinct, or
			\item consist entirely of a singular part containing exactly two singular indecomposable blocks.
		\end{enumerate}
	\end{proposition}

\begin{proof}
	Since $\Pi(y) \in \Pi(\mathfrak{R})$, Proposition 1.3 implies that every nonzero element of the pencil $\Pi(y)$ has exactly a 1-dimensional kernel on the $(2N+1)$-dimensional space $D$.

	Because the roots $a_i$ of the elementary divisors associated with the regular blocks may be complex, we must consider the complexification of the pencil. It is important to note that Proposition 1.3 remains valid for the complexified pencil $\Pi(y)^\mathbb{C}$. In particular, by \eqref{kernel_coordinate}, every nonzero element in the complexified pencil has exactly a 1-dimensional kernel.

	The reduced pencil $\mathrm{Im}\, L_\tau^\ell$ consists of the restrictions of these complexified skew-symmetric forms to the hyperplane $\tau \cap D$, which has an even dimension $2N$. Restricting a skew-symmetric form to a hyperplane changes the dimension of its kernel by exactly $+1$ or $-1$. Furthermore, the kernel of any skew-symmetric form on an even-dimensional space must have an even dimension. Therefore, since the original generic kernel dimension is 1, every nonzero element $\bar{\omega} \in \mathrm{Im}\, L_\tau^\ell$ must have a kernel of dimension exactly 0 or 2.

	We analyze the Kronecker-Weierstrass block decomposition \eqref{pair_decomp} of the basis of $\mathrm{Im}\, L_\tau^\ell$ subject to this maximum kernel dimension bound:
		\begin{itemize}
			\item A \textbf{regular block} associated with an elementary divisor parameter $a_i$ has a generic kernel of dimension 0, but exactly at the parameter line $\mu = -a_i\lambda$, its kernel dimension is at least 2 (below we use the terminology and notations of item (2) of Remark \ref{terminology_rem}).
			\item A \textbf{singular block} acts on an odd-dimensional space and thus has a kernel of dimension at least 1 for \textit{every} nonzero element of the pencil.
		\end{itemize}

	Assume the decomposition \eqref{pair_decomp} contains at least one regular block (with parameter $a_i$) and at least one singular block. Evaluating the pencil at the parameter $\mu = -a_i\lambda$ yields a form whose kernel dimension is the sum of the kernel dimensions of its constituent blocks. The regular block contributes $2$, and the singular block contributes $1$, resulting in a total kernel dimension of at least $3$. This contradicts the maximum bound of 2. Thus, the decomposition \eqref{pair_decomp} must be exclusively regular or exclusively singular.

	If the decomposition \eqref{pair_decomp} is purely regular, suppose two indecomposable blocks share the same parameter $a_i$. Evaluating the pencil at $\mu = -a_i\lambda$ causes both blocks to contribute $2$ to the kernel dimension, yielding a total kernel of dimension at least $4$. This is a contradiction, meaning all $a_i$ in a purely regular decomposition must be pairwise distinct.

	If the decomposition is purely singular, the underlying vector space dimension $2N$ dictates that there must be an even number of singular indecomposable blocks, as each singular indecomposable block operates on an odd-dimensional space. For a nonzero element in the pencil, each singular block contributes exactly $1$ to the overall kernel dimension. Thus, a pencil with $2k$ singular blocks has a generic kernel dimension of $2k$. Since the maximum allowable kernel dimension is 2, this implies the purely singular decomposition must contain exactly two singular blocks.
\end{proof}

In the next two subsections we study ampleness of $\Pi(\mathfrak R)\cap \Pi(P_\tau^\ell)$ in $\Pi(P_\tau^\ell)$ in the cases of items (1) and (2) of Proposition \ref{2cases_prop} separately.

\subsection{The case when decomposition \ref{pair_decomp} is purely regular with pairwise distinct $a_i$}
\label{ampleness-regular}

In this subsection, we prove  ampleness of $\Pi(\mathfrak R)\cap \Pi(P_\tau^\ell)$ in $\Pi(P_\tau^\ell)$ in the case of item (1) of Proposition \ref{2cases_prop}, i.e., when decomposition   \ref{pair_decomp} for a basis of the reduced pencil $\mathrm{Im}\, L_\tau^\ell$ consists of the regular part only, 
with pairwise distinct $a_i$.

Given  $y \in P_\tau^\ell$ and $\omega \in \Pi(y)$, 
let $\overline \omega:=	\mathrm{pr}_{K_e} (\omega)$, where $\mathrm{pr}_{K_e}$ as in \eqref{omegasplitting}. Let $\omega_i$, $\overline \omega_i$ , $\beta_i$ are as in \eqref{translation_to_old_1}. 
Assume that bases $\mathcal E^i$ from the decomposition \eqref{pair_decomp} are of the form
	\begin{equation}
		\label{basis_comp}
		\mathcal E^i=\{\varepsilon^{(i)}_j\}_{j=1}^{2n_i}, \quad 1\leq i\leq p.
	\end{equation}
    Furthermore, in the presence of complex indecomposable blocks in \eqref{pair_decomp} , i.e. when $p>m$, let
	\begin{equation}
		\label{basis_complex}
		\mathcal E^i:=\overline {\mathcal E^{m+i-p}}=\left\{\overline{\varepsilon^{(m+i-p)}_j}\right\}_{j=1}^{2n_i}, \quad p+1\leq i\leq 2p-m
	\end{equation}
and 
\begin{equation}
	\label{a_i_conj}
	a_i:=\overline{a_{m+i-p}}, \quad  p+1\leq i\leq 2p-m.
\end{equation}
Let $\beta^{(i)}_{k,j}$, $k=1,2$, are defined by
\begin{equation}
	\label{beta_coeff}
	\beta_k = \sum_{i=1}^{2p-m}\sum_{j=1}^{2n_i} \beta^{(i)}_{k,j} \varepsilon^{(i)}_j
\end{equation}
Note that since we work over $\mathbb R$ and based on convention given in \eqref{basis_complex} we have
\begin{equation}
	\label{beta_conj}
	\beta^{(m+i-p)}_{k,j}=\overline{\beta^{(i)}_{k,j}}, \quad p+1\leq i\leq 2p-m, k=1,2
\end{equation}
Based on this, we have the following identification
\begin{equation}
\label{beta_tuple}
\mathrm{Hom}(D^\perp, (\tau\cap D)^*)\cong\left (\{\beta_{k,j}^{(i)}\}_{k=1,2, 1\leq i\leq m, 1\leq j\leq 2n_i}, \{\operatorname{Re}\beta_{k,j}^{(i)}, \operatorname{Im}\beta_{k,j}^{(i)}\}_{k=1,2, m+1\leq i\leq p, 1\leq j\leq 2n_i}\right),
\end{equation}
more precisely, the right-hand side of \eqref{beta_tuple} forms  (linear) coordinates on $\mathrm{Hom}(D^\perp, (\tau\cap D)^*)$.

\begin{proposition}
\label{reg_ample_prop} 
Assume that the reduced pencil $\mathrm{Im}\, L_\tau^\ell$ consists of the regular part only, with pairwise distinct $a_i$.
Under identification \eqref{beta_tuple}, the following relation holds:
\begin{equation}
	\label{union}
	B\bigl(\Pi(P_\tau^\ell)\setminus \Pi(\mathfrak R)\bigr) = \bigcup_{i=1}^{2p-m} \{\beta^{(i)}_{2,1} = a_{i} \beta^{(i)}_{1,1}\} \cap \{\beta^{(i)}_{2, n_i+1} = a_i \beta^{(i)}_{1, n_i+1}\},
\end{equation}
where the map $B: \Pi(P_\tau^\ell) \to \mathrm{Hom}(D^\perp, (\tau\cap D)^*)$ is as in \eqref{B_map}.
\end{proposition}

Before proving Proposition \ref{reg_ample_prop} note that
\begin{enumerate}
	\item if $a_i$ is real, i.e. $1\leq i\leq m$, then the space in the union \eqref{union}, corresponding to index $i$
	      has codimension $2$ in $\mathrm{Hom}(D^\perp, (\tau\cap D)^*)$;
	\item if $a_i$ is nonreal, i.e. $m+1\leq i\leq 2p-m$, then by \eqref{a_i_conj} and   \eqref{beta_conj} the space in the union \eqref{union} corresponding to index $i$ is complex conjugate to the space in this union corresponding to the index $m+i-p$ in the complexification of $\mathrm{Hom}(D^\perp, (\tau\cap D)^*)$. Therefore those two spaces in the union \eqref{union} defines the space of codimension $4$ in $\mathrm{Hom}(D^\perp, (\tau\cap D)^*)$ given by the real and imaginary parts of the relations $\beta^{(i)}_{2,1} = a_{i} \beta^{(i)}_{1,1}$ and  $\beta^{(i)}_{2, n_i+1} = a_i \beta^{(i)}_{1, n_i+1}$.
\end{enumerate}

Combining this, we get that  $B\bigl(\Pi(P_\tau^\ell)\setminus \Pi(\mathfrak R)\bigr)$ is at least a codimension-two stratified subset in $\mathrm{Hom}(D^\perp, (\tau\cap D)^*)$. This implies that $\Pi(\mathfrak R)\cap \Pi(P_\tau^\ell)$  is
ample in the case where $\overline{\omega}$ is a direct sum of regular
blocks.

\begin{proof}[Proof of Proposition \ref{reg_ample_prop}]
Let $\{e^{(i)}_j\}$ be the basis of $(\tau\cap D)^{\mathbb C}$ dual to the basis $\{\varepsilon^{(i)}_j\}$ of $\bigl((\tau\cap D)^{*}\bigr)^{\mathbb C}$ formed by the coframes \eqref{basis_comp}--\eqref{basis_complex}, and, using the notation of \eqref{indecomp_phi}, put
\begin{equation}
	\label{Phi_Psi^i}
	\Phi_{2n_i}^{(i)}:=\Phi_{2n_i}^{\mathcal E^i}, \quad \Psi_{2n_i}^{(i)}:=\Psi_{2n_i}^{\mathcal E^i}, \quad 1\leq i\leq 2p-m.
\end{equation}
For $(\mu,\lambda)\in\mathbb C^2\setminus\{(0,0)\}$ write
\begin{equation}
	\label{mulambda}
	\omega_{\mu,\lambda}:=\mu\omega_1+\lambda\omega_2,\qquad
	\overline\omega_{\mu,\lambda}:=\mu\overline\omega_1+\lambda\overline\omega_2,\qquad
	\gamma_{\mu,\lambda}:=\mu\beta_1+\lambda\beta_2 ,
\end{equation}
so that, by \eqref{translation_to_old_2},
$\omega_{\mu,\lambda}=\overline\omega_{\mu,\lambda}+\gamma_{\mu,\lambda}\wedge\varepsilon$,
and, by \eqref{reg_finite} and \eqref{pair_decomp},
\begin{equation}
	\label{barmulanda}
	\overline\omega_{\mu,\lambda}=\sum_{i=1}^{2p-m}\Bigl((\mu+a_i\lambda)\,\Phi^{(i)}_{2n_i}+\lambda\,\Psi^{(i)}_{2n_i}\Bigr).
\end{equation}
By item (2) of Proposition \ref{open_orbit_char_prop}, $\Pi(y)\in\Pi(\mathfrak R)$ if and only if $\dim\ker\omega_{\mu,\lambda}=1$ for every $(\mu,\lambda)\neq(0,0)$, and by Lemma \ref{lem:kernel-extension} (applied over $\mathbb C$, with $H=(\tau\cap D)^{\mathbb C}$) the latter dimension is determined by the kernel $K$ of $\overline\omega_{\mu,\lambda}$ on $(\tau\cap D)^{\mathbb C}$ and by the restriction of $\gamma_{\mu,\lambda}$ to $K$.

The $i$-th summand of \eqref{barmulanda} is the block $E^{\mathcal E^i}_{2n_i}(a_i)$ evaluated at the parameter $(\mu,\lambda)$; it lives on the subspace spanned by $e^{(i)}_1,\dots,e^{(i)}_{2n_i}$, and these subspaces are complementary. By item (2) of Remark \ref{terminology_rem}, the determinant of the $i$-th block is a nonzero multiple of $(\mu+a_i\lambda)^{2n_i}$, so the block is nondegenerate unless $\mu=-a_i\lambda$. Consequently, if $\mu\neq-a_i\lambda$ for all $i$, then $K=0$ and $\dim\ker\omega_{\mu,\lambda}=1$ by Lemma \ref{lem:kernel-extension}(1).

It remains to consider $\mu=-a_{i_0}\lambda$ for some $i_0$; since all $a_i$ are finite, $\lambda\neq0$. As the $a_i$ are pairwise distinct, all blocks with $i\neq i_0$ are nondegenerate at this parameter, whereas the $i_0$-th block reduces to $\lambda\Psi^{(i_0)}_{2n_{i_0}}$. By \eqref{indecomp_phi},
\[
	\Psi^{(i_0)}_{2n_{i_0}}=\sum_{j=2}^{n_{i_0}}\varepsilon^{(i_0)}_j\wedge\varepsilon^{(i_0)}_{2n_{i_0}+2-j}
\]
pairs each covector with index in $\{2,\dots,n_{i_0}\}$ with exactly one covector with index in $\{n_{i_0}+2,\dots,2n_{i_0}\}$ and does not involve $\varepsilon^{(i_0)}_1$ and $\varepsilon^{(i_0)}_{n_{i_0}+1}$. Therefore
\[
	K=\ker\overline\omega_{\mu,\lambda}=\operatorname{span}\bigl\{e^{(i_0)}_1,\ e^{(i_0)}_{n_{i_0}+1}\bigr\},\qquad\dim K=2 .
\]
By Lemma \ref{lem:kernel-extension}(2), $\dim\ker\omega_{\mu,\lambda}=1$ if and only if $\gamma_{\mu,\lambda}|_K\neq0$. Now $\gamma_{\mu,\lambda}=\lambda\,(\beta_2-a_{i_0}\beta_1)$ and, by \eqref{beta_coeff}, $(\beta_2-a_{i_0}\beta_1)(e^{(i_0)}_j)=\beta^{(i_0)}_{2,j}-a_{i_0}\beta^{(i_0)}_{1,j}$; hence $\gamma_{\mu,\lambda}|_K=0$ exactly when
\[
	\beta^{(i_0)}_{2,1}=a_{i_0}\beta^{(i_0)}_{1,1}\qquad\text{and}\qquad
	\beta^{(i_0)}_{2,n_{i_0}+1}=a_{i_0}\beta^{(i_0)}_{1,n_{i_0}+1}.
\]
Letting $i_0$ run over $1,\dots,2p-m$, we obtain \eqref{union}.
\end{proof}

\subsection{The case when decomposition \ref{pair_decomp} is purely singular with two singular indecomposable blocks: reduction to the complement of the resultant of two polynomials}
\label{ampleness-singular}
In this subsection, we prove  ampleness of $\Pi(\mathfrak R)\cap \Pi(P_\tau^\ell)$ in $\Pi(P_\tau^\ell)$ in the case of item (2) of Proposition \ref{2cases_prop}, i.e. when decomposition  \ref{pair_decomp} for a basis of the reduced pencil $\mathrm{Im}\, L_\tau^\ell$ consists of the two singular indecomposable blocks $E_{2r_1+1}^{\mathcal E^{1}}(\infty)$ and $E_{2r_2+1}^{\mathcal E^{1}}(\infty)$.

Define $\mathcal E^t=\{\varepsilon_j^{(t)}\} _{j=1}^{2r_t+1}$, $t=1,2$ similar to \eqref{basis_comp} and  the tuple of coefficients 
\begin{equation}
\label{beta_coeff_sing}
\beta:=\{\beta_{k,j}^{(t)}\}_{k=1,2;t=1,2; 1\leq j\leq 2 r_t+1},
\end{equation}

as in \eqref{beta_coeff}, which is identified with $\mathrm{Hom}(D^\perp, (\tau\cap D)^*)$,
\begin{equation}
\label{beta_tuple_sing}
\mathrm{Hom}(D^\perp, (\tau\cap D)^*)\cong 
\{\beta_{k,j}^{(t)}\}_{k=1,2;t=1,2; 1\leq j\leq 2 r_t+1}.
\end{equation}
Given the tuple $\beta$, as in \eqref{beta_coeff_sing}, define the following two homogeneous polynomials in $(\mu, \lambda)$:
\begin{equation}
\label{poly_beta}
\begin{split}
P_t(\mu, \lambda;\,\beta)  =& (-1)^{r_t} \beta_{2,r_t+1}^{(t)}\lambda^{r_t+1} +\beta_{1,1}^{(t)}\mu^{r_t+1}+\\&\sum_{0 \leq m \leq r_t-1} (-1)^{r_t+m} \bigl(\beta^{(t)}_{1,r_t+1-m}-\beta^{(t)}_{2,r_t-m}\bigr)\lambda^{r_t-m} \mu^{m+1}, \quad t=1,2. 
\end{split}
\end{equation}
\begin{proposition}
\label{ample_prop_sing}
Assume that the reduced pencil $\mathrm{Im}\, L_\tau^\ell$ consists of two singular indecomposable blocks.
Under identification \eqref{beta_tuple_sing}, the set
$B\bigl(\Pi(P_\tau^\ell)\setminus \Pi(\mathfrak R)\bigr)$, where the map $B: \Pi(P_\tau^\ell) \to \mathrm{Hom}(D^\perp, (\tau\cap D)^*)$ is as in \eqref{B_map}, is equal to the set of tuples $\beta=\{\beta_{k,j}^{(t)}\}_{k=1,2;t=1,2; 1\leq j\leq 2 r_t+1}$ for which the polynomials
$P_1(\mu, \lambda;\,\beta)$ and $P_2(\mu, \lambda;\,\beta)$
vanish on a common line in $\mathbb C^2$.

\end{proposition}

\begin{proof}
As in the proof of Proposition \ref{reg_ample_prop}, let $\{e^{(t)}_j\}$ be the basis of $\tau\cap D$ dual to $\{\varepsilon^{(t)}_j\}$, put
$\Phi_{2r_t+1}^{(t)}:=\Phi_{2r_t+1}^{\mathcal E^t}$, $\Psi_{2r_t+1}^{(t)}:=\Psi_{2r_t+1}^{\mathcal E^t}$, $t=1,2$, and use the notation \eqref{mulambda}. By \eqref{reg_infinite} and our assumption on the decomposition \eqref{pair_decomp},
\begin{equation}
	\label{bar_12_sing}
	\overline{\omega}_1 = \Psi_{2r_1+1}^{(1)} + \Psi_{2r_2+1}^{(2)} \quad \text{and} \quad \overline{\omega}_2 = \Phi_{2r_1+1}^{(1)} + \Phi_{2r_2+1}^{(2)},
\end{equation}
so that, by \eqref{indecomp_phi},
\begin{equation}
	\label{bar_sing_mulambda}
	\overline\omega_{\mu,\lambda}=\sum_{t=1,2}\bigl(\lambda\Phi^{(t)}_{2r_t+1}+\mu\Psi^{(t)}_{2r_t+1}\bigr),
	\qquad
	\lambda\Phi^{(t)}_{2r_t+1}+\mu\Psi^{(t)}_{2r_t+1}
	=\sum_{j=1}^{r_t}\bigl(\lambda\varepsilon^{(t)}_j+\mu\varepsilon^{(t)}_{j+1}\bigr)\wedge\varepsilon^{(t)}_{2r_t+2-j}.
\end{equation}
Comparing dimensions in \eqref{pair_decomp} gives $2N=(2r_1+1)+(2r_2+1)$, that is,
\begin{equation}
	\label{eqn:N-r1-r2}
	N-1=r_1+r_2 .
\end{equation}

The $t$-th summand in \eqref{bar_sing_mulambda} is the pencil \eqref{singular block} on the $(2r_t+1)$-dimensional space spanned by $e^{(t)}_1,\dots,e^{(t)}_{2r_t+1}$, with $N$ replaced by $r_t$ and with the roles of $\mu$ and $\lambda$ interchanged. Hence, by \eqref{kernel_coordinate}, for every $(\mu,\lambda)\neq(0,0)$ its kernel is the line spanned by
\begin{equation}
	\label{kernel_vector_sing}
	k_t(\mu,\lambda):=\sum_{j=1}^{r_t+1}(-1)^{\,j-1}\,\mu^{\,r_t+1-j}\lambda^{\,j-1}\,e^{(t)}_j
\end{equation}
(for $r_t=0$ the block is the zero form on the line spanned by $e^{(t)}_1=k_t$). Since the two summands live on complementary subspaces,
\[
	K:=\ker\overline\omega_{\mu,\lambda}=\operatorname{span}\bigl\{k_1(\mu,\lambda),\,k_2(\mu,\lambda)\bigr\},\qquad\dim K=2,
\]
for every $(\mu,\lambda)\in\mathbb C^2\setminus\{(0,0)\}$. By Lemma \ref{lem:kernel-extension}(2), $\dim\ker\omega_{\mu,\lambda}=1$ if and only if $\gamma_{\mu,\lambda}$ does not vanish identically on $K$, that is, if and only if
\[
	\bigl(\gamma_{\mu,\lambda}(k_1(\mu,\lambda)),\ \gamma_{\mu,\lambda}(k_2(\mu,\lambda))\bigr)\neq(0,0).
\]
Since $\gamma_{\mu,\lambda}(e^{(t)}_j)=\mu\beta^{(t)}_{1,j}+\lambda\beta^{(t)}_{2,j}$ by \eqref{beta_coeff_sing}, we get
\begin{equation}
	\label{gamma_k_t}
	\gamma_{\mu,\lambda}\bigl(k_t(\mu,\lambda)\bigr)
	=\sum_{j=1}^{r_t+1}(-1)^{\,j-1}\Bigl(\beta^{(t)}_{1,j}\,\mu^{\,r_t+2-j}\lambda^{\,j-1}+\beta^{(t)}_{2,j}\,\mu^{\,r_t+1-j}\lambda^{\,j}\Bigr)
	=P_t(\mu,\lambda;\beta).
\end{equation}
Indeed, in the middle expression the monomial $\mu^{r_t+1}$ occurs only for $j=1$ in the first summand, with coefficient $\beta^{(t)}_{1,1}$; the monomial $\lambda^{r_t+1}$ occurs only for $j=r_t+1$ in the second summand, with coefficient $(-1)^{r_t}\beta^{(t)}_{2,r_t+1}$; and, for $0\le m\le r_t-1$, the monomial $\lambda^{r_t-m}\mu^{m+1}$ occurs for $j=r_t+1-m$ in the first summand and for $j=r_t-m$ in the second, with total coefficient
\[
	(-1)^{r_t-m}\beta^{(t)}_{1,r_t+1-m}+(-1)^{r_t-m-1}\beta^{(t)}_{2,r_t-m}
	=(-1)^{r_t+m}\bigl(\beta^{(t)}_{1,r_t+1-m}-\beta^{(t)}_{2,r_t-m}\bigr),
\]
in agreement with \eqref{poly_beta}.

Therefore, by item (2) of Proposition \ref{open_orbit_char_prop}, $\Pi(y)\in\Pi(\mathfrak R)$ if and only if the polynomials $P_1(\cdot;\beta)$ and $P_2(\cdot;\beta)$ have no common zero $(\mu:\lambda)\in\mathbb{CP}^1$, which is the assertion.
\end{proof}

The set of pairs of polynomials in $\mathfrak{P}(n,m)$ that share a common linear factor (i.e., vanish on a common line in $\mathbb C^2$) can be identified with the zero locus of the resultant of the pair. Classically, this resultant is computed as the determinant of the Sylvester matrix, a special $(n+m) \times (n+m)$ matrix constructed from the coefficients of these polynomials (see, for example, \cite{prasolov2004}). We denote the  values of this resultant at the pair of polynomials $(f,g)$ by $\mathrm{Res}(f,g)$. The zero locus of the resultant will be denoted by $\mathrm{Res}_0(n,m)$.

From \eqref{poly_beta}, it follows that the linear map
\begin{equation}
\label{map_to_poly}
\beta\mapsto \bigl(P_1(\mu, \lambda;\beta), P_2(\mu, \lambda;\beta)\bigr)
\end{equation}
is onto $\mathfrak{P}(n,m)$, because the linear map
\begin{equation}
    \label{b_to_b}
\Bigl\{b_{k,j}\Bigr\}_{k=1,2; 1\leq j\leq 2 r+1}\mapsto \Bigl\{b_{2,r+1},\bigl\{(b_{1,r+1-m}-b_{2,r-m}\bigr)\}_{m=0}^{r-1},b_{1,1}\Bigl\}
\end{equation}
is onto $\mathbb R ^{r+2}$. This, together with Proposition \ref{ample_prop_sing}, implies the following
\begin{corollary}
\label{cor:reduction-to-resultant} 
Assume that the reduced pencil $\mathrm{Im}\, L_\tau^\ell$ consists of two singular indecomposable blocks. The set $ \mathcal{R} \cap P_\tau^\ell$ is ample in  $P_\tau^\ell$ if and only if the set 
\begin{equation}
 \mathfrak P(r_1+1, r_2+1)\setminus \mathrm{Res}_0(r_1+1, r_2+1),   
\end{equation}
the complement to the zero locus of the resultant of $\mathfrak P(r_1+1, r_2+1)$,
is ample in $\mathfrak P(r_1+1, r_2+1)$.
\end{corollary}
\begin{proof}
By Proposition \ref{ample_prop_sing}, $B\bigl(\Pi(P_\tau^\ell)\cap\Pi(\mathfrak R)\bigr)$ is the preimage of $\mathfrak P(r_1+1,r_2+1)\setminus\mathrm{Res}_0(r_1+1,r_2+1)$ under the surjective linear map \eqref{map_to_poly}. A surjective linear map $T\colon E\to F$ between finite-dimensional vector spaces induces a bijection between the connected components of $T^{-1}(U)$ and those of $U$ for every open $U\subseteq F$, and $\operatorname{conv}\bigl(T^{-1}(C)\bigr)=T^{-1}(\operatorname{conv}C)$ for every such component $C$ (choose a linear splitting of $T$, which identifies $T^{-1}(U)$ with $U\times\ker T$). Hence $T^{-1}(U)$ is ample in $E$ if and only if $U$ is ample in $F$; combined with the linear bijection $B$ of \eqref{B_map} and with Proposition \ref{Pi_surjective}, this gives the claim.
\end{proof}
\subsection{Connected components of the complement to the resultant}
\label{sec:components}

We work with the space $\mathfrak{P}(n,m)$ of pairs $(f,g)$ of real
binary forms in the variables $(x,y)$ of degrees $n$ and $m$, as
introduced before Corollary~\ref{cor:reduction-to-resultant}; thus
$\mathfrak{P}(n,m)\cong\mathbb{R}^{n+1}\times\mathbb{R}^{m+1}$ is a space of
coefficients, and a form ``of degree $n$'' is allowed to have vanishing
leading coefficients, that is, to have $[1:0]$ among its zeros.  Zeros
of forms are always understood as points of $\mathbb{CP}^1$ taken with
multiplicity, real zeros being the ones lying on $\mathbb{RP}^1$, and
$\mathrm{Res}_0(n,m)\subset\mathfrak{P}(n,m)$ is the set of pairs whose
two binary forms have a common zero; in particular
$\mathrm{Res}_0(n,m)$ contains every pair one of whose entries vanishes
identically.

Throughout this subsection $n,m\ge 1$, which is all we need:
Corollary~\ref{cor:reduction-to-resultant} is applied with $n=r_1+1$ and
$m=r_2+1$, where $r_1,r_2\ge 0$.  (For $\min(n,m)=0$ the statements
below fail: if $m=0$, then $\mathrm{Res}(f,g)$ is a power of the constant
$g$, so $\mathfrak{P}(n,0)\setminus\mathrm{Res}_0(n,0)$ has two connected
components for every $n$.)

If $(f,g)\in\mathfrak{P}(n,m)\setminus\mathrm{Res}_0(n,m)$, then $f$ and
$g$ have in particular no common real zero, so the map
\begin{equation}\label{eq:Gamma}
  \Gamma_{f,g}\colon S^1\longrightarrow\mathbb{R}^2\setminus\{(0,0)\},
  \qquad
  \Gamma_{f,g}(x,y):=\bigl(f(x,y),\,g(x,y)\bigr),
  \qquad S^1:=\{x^2+y^2=1\},
\end{equation}
is well defined and continuous.  We set
\begin{equation}\label{eq:winding}
  w(f,g):=\deg\Gamma_{f,g}\in\mathbb{Z},
\end{equation}
the topological degree of \eqref{eq:Gamma}, i.e.\ the winding number of
$\Gamma_{f,g}$ about the origin, computed for a fixed orientation of
$S^1$.  (We say ``winding number'' rather than ``index'', the latter
being reserved in this paper for the Kronecker indices of a pencil.)

\begin{lemma}\label{lem:w-properties}
The function $w$ is locally constant on
$\mathfrak{P}(n,m)\setminus\mathrm{Res}_0(n,m)$ and satisfies:
\begin{enumerate}
  \item $w(-f,g)=w(f,-g)=-w(f,g)$, $\;w(-f,-g)=w(f,g)$, and
        $w(c_1f,c_2g)=w(f,g)$ for all $c_1,c_2>0$;
  \item if $n\not\equiv m\pmod 2$, then $w=0$;
  \item if $n\equiv m\pmod 2$, then $w(f,g)\equiv n\pmod 2$;
  \item $|w(f,g)|\le\min(n,m)$.
\end{enumerate}
\end{lemma}

\begin{proof}
Local constancy is clear, since $\Gamma_{f,g}$ depends continuously on
$(f,g)$ in the $C^0$-topology on $S^1$ and 
$\bigl(f(p),g(p)\bigr)\neq(0,0)$ for every $p\in S^1$, while the degree
is a homotopy invariant.  The same remark shows that $w$ is defined and
locally constant on the larger open set of pairs having no common zero
\emph{on $\mathbb{RP}^1$}; see Remark~\ref{rem:sigma-vs-res}.

(1) Changing the sign of $f$ (resp.\ of $g$) is post-composition of
$\Gamma_{f,g}$ with the reflection $(u,v)\mapsto(-u,v)$ (resp.\
$(u,v)\mapsto(u,-v)$), of degree $-1$; changing both signs is
post-composition with a rotation by $\pi$, of degree $+1$.  Rescaling by
$c_1,c_2>0$ is a homotopy through nonvanishing maps.

(2)--(3) Let $a\colon S^1\to S^1$, $a(p)=-p$; being the rotation by
$\pi$, it has $\deg a=+1$.  Homogeneity of $f$ and $g$ gives
\begin{equation}\label{eq:equivariance}
  \Gamma_{f,g}\circ a=\bigl((-1)^{n}f,\,(-1)^{m}g\bigr).
\end{equation}
If $n\not\equiv m\pmod 2$, the right-hand side is $\rho\circ\Gamma_{f,g}$
for a reflection $\rho$ of $\mathbb{R}^2$, so
$w=\deg(\Gamma_{f,g}\circ a)=\deg\rho\cdot w=-w$, i.e.\ $w=0$.  If
$n\equiv m\equiv 0\pmod2$, then $\Gamma_{f,g}\circ a=\Gamma_{f,g}$ by
\eqref{eq:equivariance}, so $\Gamma_{f,g}$ factors through the double
covering $S^1\to S^1/\{\pm1\}$ and $\deg\Gamma_{f,g}$ is even.  If
$n\equiv m\equiv1\pmod 2$, then $\Gamma_{f,g}\circ a=-\Gamma_{f,g}$,
i.e.\ $\Gamma_{f,g}$ is an odd map, and the degree of an odd self-map of
the circle is odd.

(4) A loop of winding number $w$ about the origin meets the line
$\{u=0\}$ at least $2|w|$ times.  The points of $S^1$ with $u=0$ are the
zeros of $f$ lying on $S^1$, and there are at most $2n$ of them, each
real zero of $f$ on $\mathbb{RP}^1$ giving a pair of antipodal points of
$S^1$; hence $2|w|\le 2n$.  Using the line $\{v=0\}$ instead gives
$|w|\le m$.
\end{proof}

The next theorem is due to Vassiliev \cite[Example~1]{Vassiliev15}; see
Remark~\ref{rem:sigma-vs-res} for the difference between his ambient
space and ours.

\begin{theorem}[Vassiliev]\label{thm:components}
Let $n,m\ge 1$.
\begin{enumerate}
  \item If $n\not\equiv m\pmod 2$, then
        $\mathfrak{P}(n,m)\setminus\mathrm{Res}_0(n,m)$ has exactly two
        connected components, distinguished by the sign of
        $\mathrm{Res}(f,g)$.  Equivalently, assuming without loss of generality that
        $n$ is even and $m$ is odd, they are distinguished by the
        parity of the number of zeros of $g$, counted with multiplicity,
        lying in the open subset
        $\{p\in\mathbb{RP}^1:\ f(p)>0\}$ of $\mathbb{RP}^1$, which is well
        defined precisely because $n$ is even.
  \item If $n\equiv m\pmod 2$, then the winding number \eqref{eq:winding}
        induces a bijection
        \[
          \pi_0\bigl(\mathfrak{P}(n,m)\setminus\mathrm{Res}_0(n,m)\bigr)
          \ \xrightarrow{\ \sim\ }\
          \bigl\{\,j\in\mathbb{Z}\ :\ |j|\le\min(n,m),\ j\equiv n\!\!\pmod 2\,\bigr\}.
        \]
        In particular, the number of connected components equals
        $\min(n,m)+1$.
\end{enumerate}
\end{theorem}

\begin{remark}\label{rem:sigma-vs-res}
Vassiliev works with the complement of the set
$\Sigma(n,m)\subset\mathfrak{P}(n,m)$ of pairs having a common zero
\emph{on $\mathbb{RP}^1$}, which is exactly the set where \eqref{eq:Gamma}
degenerates.  Since $\Sigma(n,m)\subseteq\mathrm{Res}_0(n,m)$, the two
spaces differ; the difference is immaterial here, because for $n,m\ge1$
the inclusion
$\mathfrak{P}(n,m)\setminus\mathrm{Res}_0(n,m)\hookrightarrow
 \mathfrak{P}(n,m)\setminus\Sigma(n,m)$
induces a bijection on $\pi_0$.  Indeed, a pair in
$\mathrm{Res}_0(n,m)\setminus\Sigma(n,m)$ is of one of two kinds.
Either both forms are nonzero and have a common zero in
$\mathbb{CP}^1\setminus\mathbb{RP}^1$; as the coefficients are real, they then
have a common irreducible real quadratic factor $q$, so the pair lies in
the image of $(q,f_1,g_1)\mapsto(qf_1,qg_1)$, whose source has dimension
$2+(n-1)+(m-1)=\dim\mathfrak{P}(n,m)-2$.  Or one of the two forms
vanishes identically, a condition of codimension $n+1\ge2$, respectively
$m+1\ge2$.  In either case we remove from the manifold
$\mathfrak{P}(n,m)\setminus\Sigma(n,m)$ a closed semialgebraic subset of
codimension at least $2$, which affects neither the set nor the number
of connected components.
\end{remark}

We now record the consequences of Theorem~\ref{thm:components} used in
Subsection~\ref{sec:ampleness-resultant}: the description of the
components in terms of interlacing of real zeros, and their behaviour
under the sign changes $(f,g)\mapsto(\pm f,\pm g)$.  Recall that the
real zeros of two real binary forms \emph{strictly interlace} on
$\mathbb{RP}^1$ if they are simple, equal in number, and between any two
cyclically consecutive zeros of one form there lies exactly one zero of
the other.

\begin{proposition}\label{prop:interlacing}
Assume $n\equiv m\pmod2$, $n,m\ge1$, and let
$(f,g)\in\mathfrak{P}(n,m)\setminus\mathrm{Res}_0(n,m)$.
\begin{enumerate}
  \item If $g$ has no real zeros, then $w(f,g)=0$.  If all real zeros of
        $f$ and $g$ are simple, $g$ has exactly $k\ge1$ of them on
        $\mathbb{RP}^1$, and between any two cyclically consecutive zeros of
        $g$ on $\mathbb{RP}^1$ there lies an odd number of zeros of $f$,
        then $|w(f,g)|=k$.  In particular this holds when the real zeros
        of $f$ and $g$ strictly interlace, $k$ of each.
  \item Every value $j$ with $|j|\le\min(n,m)$ and $j\equiv n\pmod2$ is
        attained: for $j\ne 0$ by a pair whose real zeros are simple and
        strictly interlace, $|j|$ of each, and for $j=0$ by a pair of
        forms without real zeros.  Consequently, by
        Theorem~\ref{thm:components}(2), the connected component labelled
        by $j$ consists exactly of the pairs with $w(f,g)=j$, and $|j|$
        is the number of interlacing real zeros of its strictly
        interlacing representatives.
  \item If $w(f,g)=j\ne0$, then $(f,g)$ and $(-f,-g)$ lie in the same
        connected component, whereas $(-f,g)$ and $(f,-g)$ lie together
        in the other component, the one labelled by $-j$; in particular
        $(f,g)$ is connected neither to $(-f,g)$ nor to $(f,-g)$.  If
        $w(f,g)=0$ --- which, for $n\equiv m\pmod 2$, can happen only
        when $n$ and $m$ are both even --- then all four pairs
        $(\pm f,\pm g)$ lie in one and the same connected component.
\end{enumerate}
\end{proposition}

\begin{proof}
(1) If $g$ has no real zeros, then $g$ has a constant sign on $S^1$, so
$\Gamma_{f,g}$ takes values in an open half plane and $w(f,g)=0$.
Assume now that the hypotheses of the second sentence of item (1) hold. The winding number of a loop
in $\mathbb{R}^2\setminus\{0\}$ is the number of its crossings of the ray
$\{(u,v)\in \mathbb R^2: u>0,\ v=0\}$ counted with signs, the crossings being transverse here
because the real zeros of $g$ are simple.  Writing $'$ for the
derivative with respect to a positively oriented parameter on $S^1$,
\begin{equation}\label{eq:crossing-formula}
  w(f,g)=\sum_{\substack{p\in S^1,\ g(p)=0\\ f(p)>0}}
          \operatorname{sign}g'(p),
\end{equation}
the sum being over the $2k$ zeros of $g$ on $S^1$ at which $f$ is
positive; note $f(p)\ne0$ at every zero $p$ of $g$, the two forms having
no common zero.  Since the zeros of $g$ are simple,
$\operatorname{sign}g'$ alternates along cyclically consecutive zeros of
$g$; by hypothesis $\operatorname{sign}f$ alternates along them as well.
Hence $\operatorname{sign}\bigl(f(p)g'(p)\bigr)$ takes one and the same
value $\varepsilon\in\{\pm1\}$ at all $2k$ zeros of $g$ on $S^1$, of
which exactly $k$ satisfy $f(p)>0$; by \eqref{eq:crossing-formula},
$w(f,g)=\varepsilon k$.  (Note that
$\operatorname{sign}\bigl(f\,g'\bigr)$ is well defined on $\mathbb{RP}^1$
and not merely on $S^1$, since under $p\mapsto-p$ it is multiplied by
$(-1)^{n+m}=1$; the sign of $w$ is thus the co-orientation of the
interlacing pattern.)

(2) For $j=0$, we have $n\equiv m\equiv 0\pmod2$. Take $f_0=(x^2+y^2)^{n/2}$ and $g_0=(x^2+2y^2)^{m/2}$,
which have no real zeros and no common zero; on $S^1$ both are positive, so
$\Gamma_{f_0,g_0}$ takes values in the  open first quadrant and $w(f_0,g_0)=0$.
For $k=|j|\ge1$ put $z=x+iy$ and
\[
  f_k:=\operatorname{Re}(z^{k})\,(x^2+y^2)^{\frac{n-k}{2}},
  \qquad
  g_k:=\operatorname{Im}(z^{k})\,(x^2+2y^2)^{\frac{m-k}{2}},
\]
real binary forms of degrees $n$ and $m$, the exponents being integers because
$k\equiv n\equiv m\pmod2$.  The pair $(f_k,g_k)$ has no common zero in
$\mathbb{CP}^1$: the zeros of $\operatorname{Re}(z^k)$ are the lines through the
$k$-th roots of $\pm i$ and those of $\operatorname{Im}(z^k)$ the lines through
the $k$-th roots of $\pm1$, which are $2k$ distinct real lines, while
$x^2+y^2$ and $x^2+2y^2$ vanish at $[\pm i:1]$ and at $[\pm i\sqrt2:1]$
respectively.  In particular the real zeros of $f_k$ and $g_k$ are simple and
strictly interlace, $k$ of each.  On $S^1$ we have
$(x^2+y^2)^{\frac{n-k}{2}}\equiv 1$, whereas
$\rho:=(x^2+2y^2)^{\frac{m-k}{2}}\big|_{S^1}=(1+y^2)^{\frac{m-k}{2}}$ is
positive; since $\operatorname{Re}(z^k)$ and $\operatorname{Im}(z^k)$ have no
common zero, the formula
$\Gamma_t(z):=\bigl(\operatorname{Re}(z^k),\,[(1-t)+t\rho(z)]\operatorname{Im}(z^k)\bigr)$,
$t\in[0,1]$, is a homotopy of maps $S^1\to\mathbb{R}^2\setminus\{(0,0)\}$ joining
$z\mapsto z^{k}$ to $\Gamma_{f_k,g_k}$, under the identification
$\mathbb{R}^2\cong\mathbb{C}$.  Hence $w(f_k,g_k)=k$, while $w(-f_k,g_k)=-k$ by
Lemma~\ref{lem:w-properties}(1).
We have thus exhibited one pair for
each of the $\min(n,m)+1$ values allowed by
Theorem~\ref{thm:components}(2), and the last assertion follows.

(3) By Lemma~\ref{lem:w-properties}(1), $w(-f,-g)=w(f,g)=j$ and
$w(-f,g)=w(f,-g)=-j$.  If $j\ne0$ these two values are distinct, so by
Theorem~\ref{thm:components}(2) the four pairs fall into the two
components labelled by $j$ and $-j$ as stated.  If $j=0$, all four pairs
have winding number $0$, and by Theorem~\ref{thm:components}(2) there is
only one component with this label; that $j=0$ forces $n$ and $m$ to be
even is Lemma~\ref{lem:w-properties}(3).
\end{proof}

\begin{remark}\label{rem:history}
The invariant \eqref{eq:winding} is a classical one in disguise: it is
the Cauchy index of $g/f$ along $\mathbb{RP}^1$, whose identification with
the signature of the B\'ezoutian quadratic form of the pair goes back to
Hermite (see \cite{KreinNaimark},
\cite[Ch.~XV, \S\S2, 9--11]{Gantmacher}, \cite[Ch.~9]{BPR}), and its
extreme values $\pm\min(n,m)$ single out the strictly interlacing pairs,
which is the classical setting of the Hermite--Biehler theorem.  The
corresponding topological classifications are due to Brockett
\cite{Brockett76}, who showed that the space of real rational transfer
functions of degree $n$ has $n+1$ connected components distinguished by
the signature of the associated Hankel matrix, and to Mostovoy
\cite[Thm.~1.3]{Mostovoy01} (see also \cite{Segal}), who showed that the
space of real rational maps $\mathbb{RP}^1\to\mathbb{RP}^1$ of degree $n$ has
$n+1$ components indexed by the topological degree.  These statements,
however, concern the case $n=m$ or a normalized slice of the coefficient
space, whereas \cite[Example~1]{Vassiliev15} is stated exactly for the
full coefficient space of pairs of real binary forms of two possibly
different degrees, that is, for $\mathfrak{P}(n,m)$; this is why we use
it here.
\end{remark}

\subsection{Proof of ampleness of the complement to the zero locus of
the resultant}
\label{sec:ampleness-resultant}

\begin{theorem}\label{thm:convex-hull}
Let $n,m\ge1$.  The convex hull of every connected component of
$\mathfrak{P}(n,m)\setminus\mathrm{Res}_0(n,m)$ is equal to the entire
$\mathfrak{P}(n,m)$, i.e.\ this complement is ample in
$\mathfrak{P}(n,m)$.
\end{theorem}

The proof uses the idea of \cite[Section~20.4, p.~176]{EliashbergMishachev2002} based on
\cite{Gromov86}, which can be described by the following lemma.

\begin{lemma}\label{lem:carath}
Assume that $\mathrm{Res}_0(n,m)$ contains a basis
$\{v_i\}_{i=1}^{\,n+m+2}$ of $\mathfrak{P}(n,m)$ such that each of the
points $\pm v_i$ lies in the closure of every connected component of
$\mathfrak{P}(n,m)\setminus\mathrm{Res}_0(n,m)$.  Then the convex hull of
every connected component of
$\mathfrak{P}(n,m)\setminus\mathrm{Res}_0(n,m)$ is the entire
$\mathfrak{P}(n,m)$.
\end{lemma}

\begin{proof}
Since $\mathrm{Res}(c_1f,c_2g)=c_1^{m}c_2^{n}\,\mathrm{Res}(f,g)$, the set
$\mathrm{Res}_0(n,m)$ is invariant under $(f,g)\mapsto(c_1f,c_2g)$ for
all $c_1,c_2\ne 0$, and each connected component of
$\mathfrak{P}(n,m)\setminus\mathrm{Res}_0(n,m)$ is invariant under this
action for $c_1,c_2>0$; in particular each connected component is a
cone.  Consequently, if $\pm v_i$ lies in the closure of a component
$P_1$, then so does $\pm cv_i$ for every $c>0$.

Let $P_1$ be a connected component and let $a_0\in\mathfrak{P}(n,m)$ be an
arbitrary point.  Since $\{v_i\}_{i=1}^{\,n+m+2}$ is a basis, the
cross-polytope $K_c:=\mathrm{conv}\{\pm cv_i\}_{i=1}^{\,n+m+2}$ is a
neighbourhood of the origin for every $c>0$, and $K_c=cK_1$; hence
$a_0$ lies in the interior of $K_c$ for $c$ large enough.  Fix such a
$c$.  By the previous paragraph, for every $\eta>0$ we may choose
points $u_i^{\pm}\in P_1$ with $\lVert u_i^{\pm}\mp cv_i\rVert<\eta$.
The convex hull of a finite set depends continuously on that set, and
$a_0$ lies at a positive distance from the boundary of $K_c$; therefore,
if $\eta$ is small enough, $a_0$ belongs to
$\mathrm{conv}\{u_i^{\pm}\}\subseteq\mathrm{conv}(P_1)$.  As $a_0$ was
arbitrary, $\mathrm{conv}(P_1)=\mathfrak{P}(n,m)$.
\end{proof}

It therefore remains to exhibit a basis of $\mathfrak{P}(n,m)$ contained
in $\mathrm{Res}_0(n,m)$ and satisfying the hypothesis of
Lemma~\ref{lem:carath}.  We take the monomial basis
\begin{equation}\label{eq:monomial-basis}
  \bigl\{\bigl(x^{i}y^{\,n-i},\,0\bigr)\bigr\}_{0\le i\le n}
  \ \cup\
  \bigl\{\bigl(0,\,x^{\,l}y^{\,m-l}\bigr)\bigr\}_{0\le l\le m},
\end{equation}
which consists of $n+m+2$ elements of $\mathrm{Res}_0(n,m)$, one of whose
two entries vanishes identically.

\begin{proposition}\label{prop:basis-in-closure}
Let $n,m\ge1$.  Each of the $2(n+m+2)$ points
\[
  \pm\bigl(x^{i}y^{\,n-i},\,0\bigr)\ \ (0\le i\le n),
  \qquad
  \pm\bigl(0,\,x^{\,l}y^{\,m-l}\bigr)\ \ (0\le l\le m)
\]
lies in the closure of \emph{every} connected component of
$\mathfrak{P}(n,m)\setminus\mathrm{Res}_0(n,m)$.
\end{proposition}

\begin{proof}
In both cases below it suffices to treat the point
$v=\bigl(x^iy^{\,n-i},0\bigr)$: the point $-v$ is treated by changing the
signs of both entries of the approximating pairs, an involution of
$\mathfrak P(n,m)\setminus\mathrm{Res}_0(n,m)$ which permutes its connected
components (it preserves each of them when $n\equiv m\pmod2$, by
Lemma~\ref{lem:w-properties}(1), and it exchanges the two components when
$n\not\equiv m\pmod2$, since then
$\mathrm{Res}(-f,-g)=(-1)^{n+m}\mathrm{Res}(f,g)=-\mathrm{Res}(f,g)$); in either
case a point lying in the closure of every component is mapped to a point with
the same property; and the points
$\pm\bigl(0,x^{\,l}y^{\,m-l}\bigr)$ are treated by exchanging the two
factors of $\mathfrak{P}(n,m)$, which changes $w$ only by a sign, the
exchange of the coordinates of $\mathbb{R}^2$ being a reflection.

\smallskip
\noindent\textbf{Case 1: $n\equiv m\pmod 2$.}
By Theorem~\ref{thm:components}(2) the connected components are labelled
by the values $j$ of the winding number with $|j|\le\min(n,m)$ and
$j\equiv n\pmod 2$.  Fix such a $j$ and put $k=|j|$.  All $n$ zeros of
$x^iy^{\,n-i}$ lie on $\mathbb{RP}^1$, so an arbitrarily small
perturbation $f_\eta$ of it has $n$ distinct simple real zeros.  If
$k=0$, then $m$ is even, and we let $g_\delta$ be an arbitrarily small
form of degree $m$ without real zeros; then $w(f_\eta,g_\delta)=0$ by
Proposition~\ref{prop:interlacing}(1).  If $k\ge1$, then $k\le n$ and
$k\equiv n\pmod2$, so we may write $n=c_1+\dots+c_k$ with all $c_a$ odd
and positive; we let $g_\delta$ be an arbitrarily small form of degree
$m$ with exactly $k$ simple real zeros (possible since $k\le m$ and
$k\equiv m\pmod2$), placed on $\mathbb{RP}^1$ so that the $a$-th gap
between cyclically consecutive zeros of $g_\delta$ contains exactly
$c_a$ zeros of $f_\eta$, and with its non-real zeros distinct from those
of $f_\eta$.  Then
$(f_\eta,g_\delta)\in\mathfrak{P}(n,m)\setminus\mathrm{Res}_0(n,m)$ and
$|w(f_\eta,g_\delta)|=k$ by Proposition~\ref{prop:interlacing}(1).
Replacing $g_\delta$ by $-g_\delta$ reverses the sign of $w$, so one of
$(f_\eta,\pm g_\delta)$ has winding number $j$ and therefore lies, by
Theorem~\ref{thm:components}(2), in the prescribed component.  Since
$\eta$ and $\delta$ may be taken arbitrarily small, $v$ lies in the
closure of that component.

\smallskip
\noindent\textbf{Case 2: $n\not\equiv m\pmod 2$.}
By Theorem~\ref{thm:components}(1) there are exactly two connected
components, distinguished by the sign of $\mathrm{Res}(f,g)$.  The
argument of Case~1 does not apply here: since
$\mathrm{Res}(-f,g)=(-1)^{m}\mathrm{Res}(f,g)$ and
$\mathrm{Res}(f,-g)=(-1)^{n}\mathrm{Res}(f,g)$, exactly one of the two
sign changes preserves the component, so changing signs does not suffice
to reach both of them.  Instead we use the classical fact that the
differential of the resultant does not vanish at a pair of binary forms
having exactly one common zero in $\mathbb{CP}^1$, that zero being a
simple zero of each of the two forms; note that over $\mathbb{R}$ such a
common zero is automatically real, since the non-real zeros of a real
form come in conjugate pairs.  In other words, the set of such pairs is
a regular stratum of the hypersurface $\mathrm{Res}_0(n,m)$; see for
example \cite[formulas (1.11) and (1.11$'$), pp.~399--400]{gkz94}, where
the gradient of the resultant at such a pair is expressed through the
common zero. (Which individual partial derivatives are nonzero depends on
the position of the common zero; only the nonvanishing of the gradient is
used below.)

Now let $f_\eta$ be an arbitrarily small perturbation of
$x^iy^{\,n-i}$ having $n$ distinct simple zeros, let $p\in\mathbb{RP}^1$
be one of its real zeros, and let $\ell_p$ be a linear form vanishing at
$p$.  Choose $g_\delta:=\delta\,\ell_p\,h$, where $h$ is a form of degree
$m-1$ having no zero in common with $f_\eta$ and not divisible by
$\ell_p$, and $\delta>0$ is arbitrarily small.  Then $p$ is the only
common zero of $f_\eta$ and $g_\delta$ and it is a simple zero of each,
so $(f_\eta,g_\delta)$ lies in the regular stratum above, and it is
arbitrarily close to $v$.  Finally, writing
$u_t:=(f_\eta,g_\delta)+t\,\nabla\mathrm{Res}(f_\eta,g_\delta)$, we get
$\mathrm{Res}(u_t)=t\,\lVert\nabla\mathrm{Res}(f_\eta,g_\delta)\rVert^2+O(t^2)$,
which is positive for small $t>0$ and negative for small $t<0$.  Hence
points of both connected components lie arbitrarily close to $v$.
\end{proof}

\begin{proof}[Proof of Theorem~\ref{thm:convex-hull}]
Immediate from Lemma~\ref{lem:carath} applied to the basis
\eqref{eq:monomial-basis}, whose hypothesis holds by
Proposition~\ref{prop:basis-in-closure}.
\end{proof}

Combining Theorem~\ref{thm:convex-hull} with
Corollary~\ref{cor:reduction-to-resultant}, we conclude that
$\mathcal{R}\cap P^{\ell}_{\tau}$ is ample in $P^{\ell}_{\tau}$ whenever
the reduced pencil $L^{\ell}_{\tau}$ consists of two singular
indecomposable blocks, which is the case of item (2) of
Proposition~\ref{2cases_prop}.
Together with Subsection~\ref{ampleness-regular} this exhausts the cases
listed in Proposition~\ref{2cases_prop}, and therefore
completes the proof of the ampleness of $\mathfrak{R}$ and hence, by
Theorem~\ref{Gromov_ample_thm}, of Theorem~\ref{main_theorem}.

\appendix

\section{Proof of Proposition \ref{prop:formal_solution}}
\label{top_sec}

Throughout this appendix, for a real vector bundle $E$ of rank $r$ we write $\det E:=\Lambda^{r}E$, and we use the canonical isomorphism
\begin{equation}
	\label{eq:contraction-iso}
	E\otimes(\det E)^{*}\xrightarrow{\ \cong\ }\Lambda^{r-1}E^{*},\qquad v\otimes\Omega\longmapsto\iota_{v}\Omega .
\end{equation}
All vector bundles are real; complexifications do not occur in this appendix.

\subsection*{Necessity ($\implies$)}
Let $(D,L)$ be a formal solution of $\mathfrak R$, i.e. an almost Kronecker structure. Put $Q:=TM/D$, so that $D^{\perp}\cong Q^{*}$ canonically, and put
\[
	A:=\det D=\Lambda^{2N+1}D .
\]

\smallskip\noindent
\emph{The span of the kernels.}
Fix $q\in M$. By item (2) of Proposition \ref{open_orbit_char_prop}, every nonzero form $L(\alpha)$, $\alpha\in D^{\perp}(q)$, has rank $2N$ on $D(q)$, so that $L(\alpha)^{N}\neq0$, and by \eqref{volume_ker} together with \eqref{eq:contraction-iso} the element of $D(q)\otimes\det D(q)^{*}$ corresponding to $L(\alpha)^{N}\in\Lambda^{2N}D(q)^{*}$ spans the line $\ker L(\alpha)\otimes\det D(q)^{*}$. Let $\mathcal V_1\subset D$ be the subbundle spanned at every point by the kernels of the nonzero forms of the pencil; by \eqref{kernel_coordinate} it has rank $N+1$, and by \eqref{singular block} it is isotropic for every form of the pencil.

Since $\alpha\mapsto L(\alpha)^{N}$ is a homogeneous polynomial map of degree $N$ on $D^{\perp}\cong Q^{*}$, it is induced by a bundle map
\begin{equation}
	\label{eq:kappa}
	\kappa\colon\mathrm{Sym}^{N}(Q^{*})\longrightarrow\Lambda^{2N}D^{*}\cong D\otimes A^{*},\qquad
	\alpha^{N}\longmapsto L(\alpha)^{N}.
\end{equation}
For a basis $\alpha_1,\alpha_2$ of $D^{\perp}(q)$, the elements $\alpha_1^{i}\alpha_2^{N-i}$, $0\le i\le N$, form a basis of $\mathrm{Sym}^{N}(Q_q^{*})$, and $\kappa$ maps them to $L(\alpha_1)^{i}\wedge L(\alpha_2)^{N-i}$, which are linearly independent by item (4) of Proposition \ref{open_orbit_char_prop}. Hence $\kappa$ is injective. Its image is contained in $\mathcal V_1\otimes A^{*}$ and, being spanned by the elements corresponding to the kernel lines, coincides with it. Comparing ranks we obtain a canonical isomorphism
\begin{equation}
	\label{eq:V1}
	\mathcal V_1\cong A\otimes\mathrm{Sym}^{N}(Q^{*}) .
\end{equation}
In particular $\mathcal V_1$ is a smooth subbundle of $D$.

\smallskip\noindent
\emph{The quotient.}
Put $\mathcal V_2:=D/\mathcal V_1$, a vector bundle of rank $N$. Since $\mathcal V_1$ is isotropic for all forms of the pencil, evaluation of these forms defines a bundle map
\begin{equation}
	\label{eq:T}
	T\colon\mathcal V_1\otimes Q^{*}\longrightarrow\mathcal V_2^{*},\qquad
	T(v\otimes\alpha)\bigl([u]\bigr):=L(\alpha)(v,u),\quad v\in\mathcal V_1,\ u\in D,
\end{equation}
which is well defined because changing $u$ by an element of $\mathcal V_1$ does not change the right-hand side.

The map $T$ is surjective. Indeed, let $u\in D(q)$ represent a class in $\mathcal V_2(q)$ annihilated by the image of $T$, so that $L(\alpha)(v,u)=0$ for all $v\in\mathcal V_1(q)$ and all $\alpha\in D^{\perp}(q)$, and fix one nonzero $\alpha$, writing $\omega:=L(\alpha)$. The subspace $\mathcal V_1(q)$ is isotropic for $\omega$ and contains $\ker\omega$, so $\mathcal V_1(q)/\ker\omega$ is an $N$-dimensional isotropic subspace of the $2N$-dimensional symplectic space $D(q)/\ker\omega$, i.e. a Lagrangian one. Therefore $\mathcal V_1(q)^{\perp_{\omega}}=\mathcal V_1(q)$, so $u\in\mathcal V_1(q)$ and $[u]=0$. Thus the annihilator of the image of $T$ is trivial, and $T$ is onto.

Next, $\ker T$ contains $A\otimes\mathrm{Sym}^{N+1}(Q^{*})$, where $\mathrm{Sym}^{N+1}(Q^{*})$ is embedded in $\mathrm{Sym}^{N}(Q^{*})\otimes Q^{*}$ by polarization, $\alpha^{N+1}\mapsto\alpha^{N}\otimes\alpha$. Indeed, under \eqref{eq:V1} an element $a\otimes\alpha^{N}$ of $\mathcal V_1$ corresponds, up to a nonzero factor, to a vector spanning $\ker L(\alpha)$, so $T(a\otimes\alpha^{N}\otimes\alpha)=0$, and the elements $\alpha^{N+1}$ span $\mathrm{Sym}^{N+1}(Q^{*})$. On the other hand, since $T$ is onto,
\[
	\operatorname{rank}\ker T=\operatorname{rank}\bigl(\mathcal V_1\otimes Q^{*}\bigr)-\operatorname{rank}\mathcal V_2^{*}=2(N+1)-N=N+2=\operatorname{rank}\mathrm{Sym}^{N+1}(Q^{*}),
\]
so $\ker T=A\otimes\mathrm{Sym}^{N+1}(Q^{*})$. For the rank two bundle $Q^{*}$ there is a canonical exact sequence
\begin{equation}
	\label{eq:sym-exact}
	0\longrightarrow\mathrm{Sym}^{N+1}(Q^{*})\longrightarrow\mathrm{Sym}^{N}(Q^{*})\otimes Q^{*}\longrightarrow\mathrm{Sym}^{N-1}(Q^{*})\otimes\Lambda^{2}Q^{*}\longrightarrow0,
\end{equation}
in which the first map is polarization and the second is
$(\xi_1\cdots\xi_N)\otimes\eta\mapsto\sum_{i=1}^{N}(\xi_1\cdots\widehat{\xi_i}\cdots\xi_N)\otimes(\xi_i\wedge\eta)$; this is the Clebsch--Gordan decomposition of $\mathrm{Sym}^{N}\otimes\mathrm{Sym}^{1}$ for $\mathrm{GL}_2$ \cite[Lecture 11]{FultonHarris}. (Exactness can also be checked directly: the second map is onto, since in a local frame $x,y$ of $Q^{*}$ it sends $x^{N-j}y^{j}\otimes y$ to $(N-j)\,x^{N-j-1}y^{j}\otimes(x\wedge y)$; its kernel has rank $2(N+1)-N=N+2$ and contains the image of the first map, which is injective and has rank $N+2$.) Tensoring \eqref{eq:sym-exact} with $A$ and comparing with $T$ we get
$\mathcal V_2^{*}\cong A\otimes\mathrm{Sym}^{N-1}(Q^{*})\otimes\Lambda^{2}Q^{*}$, i.e.
\begin{equation}
	\label{eq:V2}
	\mathcal V_2\cong A^{*}\otimes\mathrm{Sym}^{N-1}(Q)\otimes\Lambda^{2}Q .
\end{equation}
Choosing splittings of the exact sequences $0\to\mathcal V_1\to D\to\mathcal V_2\to0$ and $0\to D\to TM\to Q\to0$ and combining \eqref{eq:V1} and \eqref{eq:V2}, we obtain \eqref{TM_decomp} with $Q=TM/D$ and $A=\det D$.

\begin{remark}
	\label{rem:twist}
	The line bundle $A$ in \eqref{TM_decomp} cannot be discarded, and the reason is visible already at a single point. In the normal form \eqref{singular block}, all forms of the pencil vanish on $\mathcal V_1\times\mathcal V_1$, where $\mathcal V_1=\operatorname{span}\{e_1,\dots,e_{N+1}\}$, and on $\mathcal V_2'\times\mathcal V_2'$, where $\mathcal V_2':=\operatorname{span}\{e_{N+2},\dots,e_{2N+1}\}$. Hence for every $t\in\mathbb R^{*}$ the transformation $g_t:=t\,\mathrm{Id}_{\mathcal V_1}\oplus t^{-1}\mathrm{Id}_{\mathcal V_2'}$ preserves every form of the pencil, while it acts on $\mathcal V_1$ by the scalar $t$. Thus the stabilizer of the pencil in $\mathrm{GL}(D(q))$ acts on $\mathcal V_1(q)$ through a character which is invisible on the pencil itself, and therefore cannot be detected by the $\mathrm{GL}(Q_q^{*})$-representation theory alone; this character is $\det$, since $\det g_t=t^{N+1}t^{-N}=t$, in agreement with \eqref{eq:V1}. Note also that $A$ is trivial if and only if $D$ is orientable, so \eqref{TM_decomp} with $A$ trivial is precisely the criterion for the existence of an almost Kronecker structure with orientable $D$. In Appendix \ref{nonpar_sec} we give a closed manifold admitting an almost Kronecker structure but no decomposition \eqref{TM_decomp} with $A$ trivial.
\end{remark}

\subsection*{Sufficiency ($\impliedby$)}
Conversely, suppose that \eqref{TM_decomp} holds for some rank two bundle $Q$ and line bundle $A$, and let $D\subset TM$ be the image of the first two summands, which we denote by
\[
	\mathcal V_1:=A\otimes\mathrm{Sym}^{N}(Q^{*}),\qquad
	\mathcal V_2:=A^{*}\otimes\mathrm{Sym}^{N-1}(Q)\otimes\Lambda^{2}Q .
\]
Then $D$ has rank $(N+1)+N=2N+1$, the quotient $TM/D$ is isomorphic to $Q$, and consequently $D^{\perp}\cong Q^{*}$. (The choice of $A$ is consistent with the necessity part: using $\det\mathrm{Sym}^{k}(Q)\cong(\det Q)^{\otimes k(k+1)/2}$ one finds $\det\mathcal V_1\cong A^{\otimes(N+1)}\otimes(\det Q^{*})^{\otimes N(N+1)/2}$ and $\det\mathcal V_2\cong A^{\otimes(-N)}\otimes(\det Q)^{\otimes N(N+1)/2}$, so that $\det D\cong A$.)

We define $L\colon Q^{*}\to\Lambda^{2}D^{*}$ by requiring $\mathcal V_1$ and $\mathcal V_2$ to be isotropic for every $L(\alpha)$, so that $L(\alpha)$ is determined by its cross-terms on $\mathcal V_1\times\mathcal V_2$, and by setting
\begin{equation}
	\label{eq:L-def}
	L(\alpha)\bigl(a\otimes p,\ \varphi\otimes u\otimes\Delta\bigr):=\varphi(a)\,\Delta\bigl(\iota_{u}p,\ \alpha\bigr),
\end{equation}
where $a\in A$, $\varphi\in A^{*}$, $p\in\mathrm{Sym}^{N}(Q^{*})$, $u\in\mathrm{Sym}^{N-1}(Q)$, $\Delta\in\Lambda^{2}Q$.
Here $\iota_{u}p\in Q^{*}$ is the contraction of the $N-1$ contravariant indices of $u$ with $N$ of the covariant indices of $p$; in a local frame it is the action of $u$, viewed as a constant-coefficient differential operator of order $N-1$, on the homogeneous polynomial $p$ of degree $N$. The bivector $\Delta$ is viewed as a skew-symmetric bilinear form on $Q^{*}$. All tensor factors in \eqref{eq:L-def} are contracted through their natural pairings, so $L$ is a well-defined bundle map.

To see that $L(D^{\perp}(q))$ is in the generic orbit at every point, choose a local generator $a$ of $A$ with dual generator $a^{*}$ of $A^{*}$, a local frame $x,y$ of $Q^{*}$ with dual frame $X,Y$ of $Q$, and put $\Delta:=X\wedge Y$. Define local frames of $\mathcal V_1$ and $\mathcal V_2$ by
\begin{equation}
	\label{eq:weight-frame}
	e_k:=a\otimes\frac{x^{N-k+1}y^{k-1}}{(N-k+1)!\,(k-1)!},\quad 1\le k\le N+1,
	\qquad
	e_{2N+2-k}:=a^{*}\otimes X^{N-k}Y^{k-1}\otimes\Delta,\quad 1\le k\le N,
\end{equation}
and let $\varepsilon_1,\dots,\varepsilon_{2N+1}$ be the dual coframe of $D$. The differential operator $X^{N-k}Y^{k-1}=\partial_x^{N-k}\partial_y^{k-1}$ sends the $k$-th divided power monomial in \eqref{eq:weight-frame} to $x$, the $(k+1)$-st to $y$, and all the other ones to zero. Hence, writing $\mu:=\Delta(x,\alpha)$ and $\lambda:=\Delta(y,\alpha)$ for $\alpha\in Q^{*}$, we get from \eqref{eq:L-def}
\[
	L(\alpha)(e_k,e_{2N+2-k})=\mu,\qquad L(\alpha)(e_{k+1},e_{2N+2-k})=\lambda,\qquad 1\le k\le N,
\]
and all other values of $L(\alpha)$ on pairs of frame vectors vanish. Thus
\[
	L(\alpha)=\sum_{k=1}^{N}(\mu\varepsilon_k+\lambda\varepsilon_{k+1})\wedge\varepsilon_{2N+2-k},
\]
and as $\alpha$ varies over $Q^{*}_q$, the pair $(\mu,\lambda)$ varies over $\mathbb R^{2}$; so $L(D^{\perp}(q))$ is the pencil \eqref{singular block}. Hence $L$ is fiberwise injective and its image belongs to the generic orbit at every $q\in M$, i.e. $(D,L)$ is an almost Kronecker structure. This completes the proof of Proposition \ref{prop:formal_solution}.

\section{Example of Non-Parallelizable Manifolds with desired corank $2$ distributions}
\label{nonpar_sec}
The bundle decomposition \eqref{TM_decomp}
does not imply that $M$ is parallelizable, even when $M$ is closed
and orientable. We give an explicit nine-dimensional example,
followed by a five-dimensional nonorientable example; in both of them the line bundle $A$ of \eqref{TM_decomp} is trivial. We then give a seven-dimensional example in which \eqref{TM_decomp} holds only with a nontrivial $A$, so that the line bundle in Proposition \ref{prop:formal_solution} cannot be dispensed with.
All vector bundles below are real, and $\varepsilon^r$ denotes
the trivial bundle of rank $r$ over the relevant base.

We use the standard identity
\begin{equation}
	\label{eq:nonparallelizable-projective-tangent}
	T\mathbb{RP}^{m}\oplus\varepsilon^1
	\cong \gamma^{\oplus(m+1)},
\end{equation}
where $\gamma$ is the tautological line bundle over
$\mathbb{RP}^{m}$. Indeed, writing $\gamma^\perp$ for the
orthogonal complement of $\gamma$ in $\varepsilon^{m+1}$, we have
\[
	T\mathbb{RP}^{m}\cong\operatorname{Hom}(\gamma,\gamma^\perp).
\]
Since $\operatorname{Hom}(\gamma,\gamma)\cong\varepsilon^1$
and $\gamma^*\cong\gamma$, it follows that
\[
	T\mathbb{RP}^{m}\oplus\varepsilon^1
	\cong
	\operatorname{Hom}(\gamma,\gamma^\perp\oplus\gamma)
	\cong
	\operatorname{Hom}(\gamma,\varepsilon^{m+1})
	\cong
	\gamma^{\oplus(m+1)}.
\]

\medskip
\noindent
\textbf{A closed orientable example.}
Take
\[
	N=3,
	\qquad
	M=\mathbb{RP}^{5}\times T^4,
	\qquad
	Q=\lambda\oplus\lambda,
	\qquad
	A=\varepsilon^1,
\]
where $T^4=(S^1)^4$ and
$\lambda=\pi_1^*\gamma$ is the pullback of the tautological line
bundle under the projection
$\pi_1\colon M\to\mathbb{RP}^{5}$.
The dimension of $M$ is $9=2N+3$.
Since the tangent bundle of the torus is trivial,
\eqref{eq:nonparallelizable-projective-tangent} gives
\begin{equation}
	\label{eq:nonparallelizable-nine-dimensional-tangent}
	\begin{aligned}
		TM
		 & \cong \pi_1^*T\mathbb{RP}^{5}\oplus\varepsilon^4 \\
		 & \cong
		\pi_1^*\bigl(T\mathbb{RP}^{5}\oplus\varepsilon^1\bigr)
		\oplus\varepsilon^3                                 \\
		 & \cong \lambda^{\oplus6}\oplus\varepsilon^3.
	\end{aligned}
\end{equation}
In particular, this is an actual bundle isomorphism, rather than
merely a stable one.

For any real line bundle $\lambda$, a choice of bundle metric gives
$\lambda^*\cong\lambda$ and
$\lambda^{\otimes2}\cong\varepsilon^1$.
Moreover, since $Q\cong\lambda\otimes\varepsilon^2$,
\[
	\operatorname{Sym}^{j}(Q)
	\cong
	\lambda^{\otimes j}\otimes
	\operatorname{Sym}^{j}(\varepsilon^2).
\]
The bundle $\operatorname{Sym}^{j}(\varepsilon^2)$ is trivial
of rank $j+1$. Consequently,
\[
	\operatorname{Sym}^{3}(Q^*)\cong\lambda^{\oplus4},
	\qquad
	\operatorname{Sym}^{2}(Q)\cong\varepsilon^3,
	\qquad
	\Lambda^2Q\cong\varepsilon^1.
\]
Combining these identities with
\eqref{eq:nonparallelizable-nine-dimensional-tangent}, we obtain
\[
	\begin{aligned}
		 & \operatorname{Sym}^{3}(Q^*)
		\oplus
		\bigl(\operatorname{Sym}^{2}(Q)\otimes\Lambda^2Q\bigr)
		\oplus Q                       \\
		 & \qquad\cong
		\lambda^{\oplus4}\oplus\varepsilon^3\oplus\lambda^{\oplus2}
		\cong
		\lambda^{\oplus6}\oplus\varepsilon^3
		\cong TM.
	\end{aligned}
\]
Thus \eqref{TM_decomp} holds
with $N=3$ and $A$ trivial.

To see that $M$ is not parallelizable, let
$a=w_1(\lambda)\in H^1(M;\mathbb{Z}/2)$.
The class $a^2$ is nonzero, since its restriction to
$\mathbb{RP}^{5}\times\{\mathrm{pt}\}$ is the nonzero
degree-two class in
\[
	H^*(\mathbb{RP}^{5};\mathbb{Z}/2)
	\cong(\mathbb{Z}/2)[x]/(x^6),
	\qquad |x|=1.
\]
The Whitney product formula and
\eqref{eq:nonparallelizable-nine-dimensional-tangent} imply
\[
	w(TM)=(1+a)^6.
\]
Hence
\[
	w_1(TM)=6a=0,
	\qquad
	w_2(TM)=\binom{6}{2}a^2=a^2\ne0.
\]
Therefore $M$ is orientable but not parallelizable.

\medskip
\noindent
\textbf{A five-dimensional closed nonorientable example.}
Take
\[
	N=1,
	\qquad
	M=\mathbb{RP}^{2}\times T^3,
	\qquad
	Q=\varepsilon^1\oplus\lambda,
	\qquad
	A=\varepsilon^1,
\]
where $\lambda$ is now the pullback of the tautological line bundle
over $\mathbb{RP}^{2}$.
Again, \eqref{eq:nonparallelizable-projective-tangent} gives
\[
	TM\cong\lambda^{\oplus3}\oplus\varepsilon^2.
\]
On the other hand,
\[
	Q^*\cong\varepsilon^1\oplus\lambda,
	\qquad
	\Lambda^2Q\cong\lambda.
\]
Since $\operatorname{Sym}^{0}(Q)\cong\varepsilon^1$, we obtain
\[
	\begin{aligned}
		\operatorname{Sym}^{1}(Q^*)
		\oplus
		\bigl(\operatorname{Sym}^{0}(Q)\otimes\Lambda^2Q\bigr)
		\oplus Q
		 & \cong Q^*\oplus\Lambda^2Q\oplus Q          \\
		 & \cong \varepsilon^2\oplus\lambda^{\oplus3} \\
		 & \cong TM.
	\end{aligned}
\]
Thus \eqref{TM_decomp} also holds
in this case, again with $A$ trivial.
Writing $a=w_1(\lambda)$, we have
\[
	w_1(TM)=3a=a\ne0.
\]
This manifold is therefore nonorientable and, in particular,
nonparallelizable.

\medskip
\noindent
\textbf{An example where the line bundle $A$ is essential.}
Take
\[
	N=2,\qquad M=\mathbb{RP}^{4}\times T^{3},\qquad Q=\varepsilon^{2},\qquad A=\lambda,
\]
where $\lambda$ is the pullback of the tautological line bundle over $\mathbb{RP}^4$. The dimension of $M$ is $7=2N+3$, and \eqref{eq:nonparallelizable-projective-tangent} gives
\[
	TM\cong\lambda^{\oplus5}\oplus\varepsilon^{2}.
\]
Since $Q$ is trivial, $A\otimes\operatorname{Sym}^2(Q^*)\cong\lambda^{\oplus3}$ and $A^*\otimes\operatorname{Sym}^1(Q)\otimes\Lambda^2Q\cong\lambda^{\oplus2}$ (recall that $\lambda^*\cong\lambda$), so the right-hand side of \eqref{TM_decomp} is $\lambda^{\oplus3}\oplus\lambda^{\oplus2}\oplus\varepsilon^2\cong TM$. Hence $M$ admits an almost Kronecker structure and, by Corollary \ref{existence}, a distribution satisfying $\mathfrak R$; the distribution $D\cong\lambda^{\oplus5}$ so obtained is nonorientable.

We claim that no rank two bundle $Q'$ over $M$ satisfies \eqref{TM_decomp} with $A$ trivial, so that, by Remark \ref{rem:twist}, $M$ carries no almost Kronecker structure with orientable $D$. Let $a':=w_1(Q')$, $b':=w_2(Q')$ and
\[
	F:=\operatorname{Sym}^{2}(Q'^{*})\oplus\bigl(Q'\otimes\Lambda^{2}Q'\bigr)\oplus Q'.
\]
By the splitting principle, if $Q'$ has formal roots $\xi_1,\xi_2$ (so that $a'=\xi_1+\xi_2$ and $b'=\xi_1\xi_2$), then $\operatorname{Sym}^2(Q'^*)$ has formal roots $2\xi_1,\ \xi_1+\xi_2,\ 2\xi_2$ and $Q'\otimes\Lambda^2Q'$ has formal roots $\xi_2,\xi_1$; since we work with $\mathbb Z_2$-coefficients, $2\xi_i=0$. Hence $w(\operatorname{Sym}^2(Q'^*))=1+a'$, $w(Q'\otimes\Lambda^2Q')=1+a'+b'$, and
\[
	w(F)=(1+a')(1+a'+b')^2=(1+a')(1+a'^2+b'^2),
\]
so that $w_1(F)=a'$ and $w_2(F)=a'^2=w_1(F)^2$. On the other hand, $w(TM)=(1+a)^5$ with $a:=w_1(\lambda)$, whence $w_1(TM)=a$ and $w_2(TM)=\binom52a^2=0$, while $a^2\neq0$ (its restriction to $\mathbb{RP}^4\times\{\mathrm{pt}\}$ is the nonzero class $a^2\in H^2(\mathbb{RP}^4;\mathbb Z_2)$). Thus $w_2(TM)\neq w_1(TM)^2$, and $F\cong TM$ is impossible for every $Q'$. This proves the claim.

\medskip
In each of the three examples, choose a bundle isomorphism
\[
	\Phi\colon
	\bigl(A\otimes\operatorname{Sym}^{N}(Q^*)\bigr)
	\oplus
	\bigl(A^*\otimes\operatorname{Sym}^{N-1}(Q)\otimes\Lambda^2Q\bigr)
	\oplus Q
	\longrightarrow TM
\]
and let $D\subset TM$ be the image of the first two summands.
Then $D$ is a smooth corank two subbundle and
\[
	TM/D\cong Q.
\]
Equivalently, its annihilator $D^\perp\subset T^*M$ satisfies
$D^\perp\cong Q^*$.
Thus all three examples realize the quotient-bundle identification
as well as the required tangent-bundle decomposition.

\section{ The case of corank 2 distributions on even-dimensional manifolds} 
\label{even_dim_sec}
For completeness we briefly indicate what the method gives for corank $2$ distributions of \emph{even} rank $2N$ on a $(2N+2)$-dimensional manifold, $N\ge2$. There the Levi pencil lives on an even-dimensional space and its Pfaffian
\[
	P(\mu,\lambda):=\operatorname{Pf}(\mu\omega_1+\lambda\omega_2)
\]
is a real binary form of degree $N$ whose roots in $\mathbb{CP}^1$ are the elementary divisors of the pencil. A generic pencil is regular with $N$ distinct elementary divisors, but the number $k\equiv N \pmod 2$ of real ones is an invariant (and for $N\ge4$ the elementary divisors themselves carry continuous moduli), so there is no single generic orbit; the natural open $\mathrm{Diff}$-invariant relations are $\mathfrak R^{\mathrm{reg}}=\{P\not\equiv0\}$ and, for each admissible $k$, the relation $\mathfrak R_k=\{P\text{ has exactly }k\text{ real roots, all simple}\}$; for $N=2$, $\mathfrak R_2$ and $\mathfrak R_0$ are the hyperbolic and elliptic $(4,6)$ distributions of \cite{MAdP2021}. The reduction of Subsection \ref{preliminaries} applies verbatim: on a principal subspace $P^\ell_\tau$ with $D\not\subset\tau$, writing $\omega_i=\overline\omega_i+\beta_i\wedge\varepsilon$ as in \eqref{translation_to_old_2} and $\overline\omega_{\mu,\lambda}^{\,N-1}=\iota_{X(\mu,\lambda)}\mathrm{vol}_{\tau\cap D}$ with $X$ a vector-valued binary form of degree $N-1$, one finds, since $\overline\omega_{\mu,\lambda}^{\,N}=0$ for dimensional reasons,
\[
	P(\mu,\lambda)=\bigl\langle\mu\beta_1+\lambda\beta_2,\,X(\mu,\lambda)\bigr\rangle ,
\]
so that the free parameters $(\beta_1,\beta_2)$ enter \emph{one} binary form linearly, instead of the pair $(P_1,P_2)$ of \eqref{poly_beta}. Two conclusions follow at once. First, $\mathfrak R^{\mathrm{reg}}$ is ample: its trace on $P^\ell_\tau$ is either empty or the complement of a linear subspace of codimension at least two, so $\mathfrak R^{\mathrm{reg}}$ satisfies the multi-parametric $C^0$-close $h$-principle. Second, none of the relations $\mathfrak R_k$ is ample. For $k=0$ ($N$ even) the condition ``$P$ definite'' is convex. For $k\ge1$ choose $\tau$ so that $\tau\cap D$ contains the two-dimensional block of a real elementary divisor $[-a:1]$; then $P(\mu,\lambda)=(\mu+a\lambda)\,g(\mu,\lambda)$ with $g$ linear in $(\beta_1,\beta_2)$, the real root $[-a:1]$ is frozen, and its simplicity $g(-a,1)\neq0$ is a nontrivial \emph{linear} condition, so the trace of $\mathfrak R_k$ lies in the complement of a hyperplane and its components have convex hulls in half-spaces. The same principal subspaces show that any open condition bounding the number of real roots of $P$ from below by $k\ge2$ fails to be ample. This is consistent with \cite{MAdP2021}, where the hyperbolic $(4,6)$ relation is shown to be non-ample and the $h$-principle is nevertheless obtained by convex integration with avoidance; the computation above identifies, for all $N$ and $k\ge1$, the principal subspaces to be avoided as those whose hyperplane $\tau\cap D$ contains a whole real regular block. Finally, the odd-rank relation $\mathfrak R$ of the present paper has no even-rank analogue in the naive sense: it is the generic orbit of \emph{singular} pencils, and on an even-dimensional space singularity, $P\equiv0$, is a closed condition.

\subsection*{Acknowledgements}

We are deeply grateful to Álvaro del Pino and Javier Martínez-Aguinaga for many discussions at the initial stage of this project; their expertise in the $h$-principle was instrumental at the
initial stage of the project: they introduced us to the basic techniques of
convex integration used here and shaped our approach to the problem. 

In the preparation of this paper the authors used the large language models
Anthropic (model Claude Fable~5.1) and OpenAI  (model ChatGPT 6~Astra)
as editorial assistants. In particular, these tools helped to improve the
exposition, pointed out an oversight in an earlier formulation of item (3) of
Proposition~\ref{ample_dim_1}, observed that the line bundle $A=\det D$ in
Proposition~\ref{TM_decomp} cannot be omitted, and suggested proving
Propositions~\ref{reg_ample_prop} and~\ref{ample_prop_sing} via the kernel
condition~\eqref{1dim_ker} of Proposition~\ref{open_orbit_char_prop} rather
than via the equivalent condition~\eqref{wedge_power} on exterior powers,
which considerably shortened our original proofs. All results and proofs in
the paper were formulated, verified and are the sole responsibility of the
authors.

\end{document}